\documentclass[12pt]{amsart}

\usepackage{xcolor}

\title{Linear response for Bernoulli Convolutions}
\author{Jianning Fu}
\date{\today}

\usepackage[letterpaper,margin=1in]{geometry}
\usepackage{xcolor}
\usepackage{graphicx}
\usepackage{amsmath}
\usepackage{amsthm}
\usepackage{amsfonts}
\usepackage{mathrsfs}
\usepackage{amssymb} 
\usepackage{fancyhdr}
\usepackage{enumerate}
\usepackage{tgschola} 
\usepackage
{biblatex}
\renewbibmacro{in:}{}
\AtEveryBibitem{%
  \clearfield{number}%
  \clearfield{doi}%
  \clearfield{url}%
  \clearfield{isbn}%
  \clearfield{issn}%
  \clearfield{note}%
  \clearlist{publisher}%
  \clearlist{location}%
}
\usepackage[T1]{fontenc}
\usepackage[utf8]{inputenc}
\usepackage{xspace}

\newtheorem{thm}{Theorem}[section]

\newtheorem{df}{Definition}[section]

\newtheorem{prop}{Proposition}[section]
\newtheorem{lm}[prop]{Lemma}

\newtheorem{que}{Question}[section]

\numberwithin{equation}{section}

\newcommand{\R}{\mathbb{R}}
\newcommand{\E}{\mathbb{E}}
\newcommand{\br}[1]{\langle #1\rangle}

\newcommand{\N}{\mathbb N}
\newcommand{\Z}{\mathbb Z}

\newcommand{\dd}{\,\mathrm d}
\newcommand{\ii}{\mathrm i}

\newcommand{\supp}{\operatorname{supp}}
\newcommand{\Leb}{\operatorname{Leb}}
\newcommand{\dist}{\operatorname{dist}}

\newcommand{\cV}{\mathcal V}
\newcommand{\eps}{\varepsilon}
\newcommand{\DS}{\displaystyle}

\def\blue{\color{blue}}

\usepackage[
colorlinks,
breaklinks,
unicode
]{hyperref}

\begin{document}
\maketitle	
\begin{abstract}
 Let $\mu_{\lambda}$ be the Bernoulli convolution measure with parameter $\lambda\in(0,1)$. We study the regularity of 
 the function
 $h=h_{\phi}:\lambda\mapsto \int_{\mathbb{R}}\phi(x)\,d\mu_{\lambda}(x)$ 
 for H\"older observables $\phi$. We describe sufficient conditions for both smoothness and non smoothness of this function.
In particular, we  show that if $\phi$ is taken to be the classical Weierstrass function, then $h_\phi$ exhibits a phase transition: it is almost nowhere differentiable for $\lambda$ below some threshold, and it is almost everywhere differentiable for large $\lambda$ above.
 
\end{abstract}

	\bigskip

\section{Introduction}\label{Section: Introduction}
 Consider the random variable
$$X(\lambda) := \sum_{m=0}^\infty a_m \lambda^{m}, \lambda\in(0,1)$$
where $a_1,a_2,\dots$ are i.i.d random variables taking values in $\{-1,1\}$ with probability $(\frac{1}{2},\frac{1}{2})$. Denote by $\mu_\lambda$ the law of $X(\lambda)$. It is the infinite convolution product of $\DS \frac{1}{2}(\delta_{-\lambda^n}+\lambda_{\lambda^n})$, and is hence termed (unbiased)
{\it infinite Bernoulli Convolution.}  The study of Bernoulli Convolution has been an active research  area since the $1930$'s,  and numerous results have been established regarding its dimension, singularity vs. absolute continuity, $L^q$-density, multifractal spectrum, and many other aspects. Some classical works include \cite{erdos1939family} \cite{feng2023dimension} \cite{hochman2014self}  \cite{solomyak1995random}  \cite{simon2002dimension}  \cite{shmerkin2016absolute}  \cite{feng2012multifractal}, but this list is far from comprehensive.

In the present work we study  weak smoothness of the map $\lambda\mapsto \mu_{\lambda}$, which we call 
{\it the linear response problem} for Bernoulli convolution. Let $\phi(x):\mathbb{R}\to \mathbb{R}$ be a measurable function. Define 
\begin{equation}
\label{Average}
h(\lambda) = h_\phi (\lambda) := \int\phi(x)\,d\mu_{\lambda}(x),
\end{equation}
whenever the integral is well-defined and finite. We would like to study how the smoothness of $h$ depends on $\lambda$ and the H\"older (or Sobolev) regularity of $\phi$.  

The linear response theory originated in Ruelle's work and is motivated by applications to homogenization and mathematical physics, see \cite{ruelle1997differentiation, BDL00, 
BLRB00, dolgopyat05, CD06, CD09, grigo19}.
Currently,there is a well-developed linear response theory for many smooth  or piecewise smooth dynamical systems, see 
 \cite{ruelle1997differentiation, KL99,
dolgopyat2004differentiability, gouezel2006banach, baladi2008linear, baladi2007anisotropic, ruelle08, ruelle09, ruelle12, BT16, Zhang18, BRS20, SR21, Klo22, SVHV, BG24, AB-BC}
and the references therein. We also refer the readers to 
\cite{baladielse, BBS15, dLS18} for examples of failure of linear response.

Most papers above discuss the smoothness of averages such as \eqref{Average} for smooth observables $\phi.$ On the other hand in applications one often has to deal dynamically defined observables
such as the directions of the invariant splitting, and those observables are only H\"older in general, \cite{HK90, HW99, HNW02, XuZhang}. This motivates the study of regularity of
\eqref{Average} for H\"older functions.

Another reason for us to study H\"older observables is that if $\phi$ is smooth, trivially \eqref{Average} is smooth. Indeed, $X(\lambda)$ is an analytic function in $\lambda$ and the image of $X(\lambda)$ is uniformly bounded on any compact subset of $[0, 1)$. So, if $\phi$ is differentiable, then $h(\lambda)$ is automatically differentiable. We will discuss this in more detail in \S \ref{Subsection: Proof of Main Proposition}. It is therefore natural to consider observables with worse regularity like H\"older functions, which turns out to be more subtle. In this paper we present several results demonstrating both regularity of \eqref{Average} for some values of the parameters $\lambda$ and the lack thereof for other parameters. In particular, our results, shows, the phase transition: as $\lambda$ increases the averages seem to become more regular, compare Theorems \ref{Main Theorem: Differentiablility of Bernoulli Convolution with Holder Test Function} and 
\ref{thm:phase transition} below. \\

Next, we provide precise formulations of our results.

\begin{thm}\label{Main Theorem: Differentiablility of Bernoulli Convolution with Holder Test Function}
    For any $\epsilon>0$, there exist $\gamma = \gamma(\epsilon)>0$ so that for any H\"older continuous test function $\phi$ with H\"older exponent $>1-\gamma$, the  average $h_\phi$ is absolutely continuous
    on the interval $(2^{-\frac{1}{2}}+\epsilon,1 -\epsilon)$. 
\end{thm}

 The proof of Theorem \ref{Main Theorem: Differentiablility of Bernoulli Convolution with Holder Test Function} uses Fourier decay and it can be of independent interest in a more general setting.

To state our next result let $\dim \mu_{\lambda_0}$ stand for the (upper) Hausdorff dimension of $\mu_{\lambda_0}$: $\dim\mu_{\lambda_0} = \inf\{\dim_H (A): \mu_{\lambda_0}(A) = 1\}$, where $\dim_H$ denotes the Hausdorff dimension
(see Section \ref{Section: Counter Example due to Dimension Drop} for precise definition of the Hausdorff dimension).

\begin{thm}\label{Main Thorem: Counter Example due to Dimension Drop}
If $\lambda_0\in (0,1)$ is such that $\dim \mu_{\lambda_0}<1$, then for any $\theta\in(0,1)$, $h_\phi$ is not differentiable at $\lambda_0$ for generic $\phi\in C^\theta$ (the space of  H\"older continuous functions with H\"older exponent $\theta$). 
\end{thm}

The above two results draw a connection between differentiability of \eqref{Average} and the Hausdorff dimension $\dim \mu_{\lambda}$ of $\mu_{\lambda}$. The latter  has long been  the central topic in the study of Bernoulli convolution. It is well known that $\dim\mu_{\lambda}\le \frac{\log2}{-\log\lambda}$ \cite[Chapter~3,10]{d013e7c057ce4b198906d18f73a2536b}, and equality holds if $\lambda\in(0,\frac{1}{2}]$. For $\lambda \in(\frac{1}{2},1)$, according to Hochman\cite{hochman2014self} and Varju\cite{varju2019dimension}, $\dim\mu_{\lambda} = 1$ except possibly at a certain class of algebraic parameters that satisfy the ``exact overlap assumption''. Moreover, for the moment the only known examples in this class are the reciprocals of Pisot numbers
\cite{garsia1963entropy} \cite{feng2004bernoulli} and reciprocals of ``quadratic Salem numbers''  claimed by OpenAI\cite{OpenAI2026Bernoulli}. Among all these parameters of dimension drop, the largest one is $\psi^{-1} \approx \frac{1}{1.3247}\approx 0.7548$ \cite{siegel1944algebraic}. And it is conjectured that for all $\lambda$ sufficiently close to $1$, $\mu_\lambda$ should have dimension $1$. Note that $ \psi^{-1}>2^{-\frac{1}{2}}$, so by Theorem \ref{Main Thorem: Counter Example due to Dimension Drop} one can not expect actual differentiability on the entire $(2^{-\frac{1}{2}} +\epsilon, 1)$ of $h(\lambda) = \int_\R \phi(x)\,d\mu_\lambda(x)$, when $\phi$ is merely an arbitrary H\"older continuous. On the other hand, if one could find $\lambda_0$ and upgrade the conclusion of Theorem \ref{Main Theorem: Differentiablility of Bernoulli Convolution with Holder Test Function} to everywhere differentiability on $(\lambda_0,1)$ for all sufficiently H\"older observable, then $\dim_\lambda\mu_\lambda = 1$ for every $\lambda\in (\lambda_0,1)$.

Given the above connection between dimension drop and failure of linear response, it is natural to try to find sharper characterizations of success of linear response. Below is a result that partially answers this question. Namely, if one takes the classical Weierstrass function 

 $$ W_{a,b}(x)=\sum_{n\ge0}a^n\cos(2\pi b^nx) , \quad0<a<1<b,ab>1$$
 
as observable, then there is a strict phase transition:

\begin{thm}\label{thm:phase transition}
Let $a,b,W_{a,b}$ be as above and $h(\lambda) = \int W_{a,b}(x)\,d\mu_\lambda(x)$. Let 
\begin{equation}
\label{DefPar}
 \alpha=\frac{-\log a}{\log b}\in(0,1),\qquad
 d(\lambda)=\frac{\log2}{-\log\lambda},\qquad
 \lambda_c=2^{-1/(1-\alpha)}.
\end{equation}

Then $h$ is differentiable  at Lebesgue-almost every $\lambda\in(\lambda_c,1)$, while for Lebesgue-almost every $\lambda\in(0,\lambda_c)$,
\begin{equation}
 \limsup_{t\to0}
 \frac{|h(\lambda+t)-h(\lambda)|}{|t|}=\infty,
                                                        \label{eq:main}
\end{equation}
and in particular $h$  has no finite ordinary derivative there. 
\end{thm}

Note that $\alpha$ is the H\"older exponent of $W_{a,b}$, and the transition point $\lambda_c$ is such that 
$$\alpha+d(\lambda) = 1.$$
Still, the above theorem concludes \textit{a.e.} (non)differentiability, which is far from sharp. 

Apart from studying pointwise differentiability of \eqref{Average}, one can study other notions of convergence of  $\DS \lim_{\eps\to0}\frac{\mu_{\lambda+\epsilon} - \mu_{\lambda}}{\epsilon}$ converges in stronger senses, like under the $W_q$ metric or converges in the dual of some function space, which is usually termed as the ``uniform linear response'' problem. We do not address this type of convergence, and refer the interested readers to \cite{alves2026linearresponseskewproductmaps} \cite{KLOECKNER_2022}. Among other things, in the first article it is shown that $\lim_{\epsilon\to 0} \frac{\mu_{\lambda+\epsilon} - \mu_{\lambda}}{\epsilon}$  exist in the dual of $C^3$ functions, and in the second it is shown that $\lambda\mapsto\mu_\lambda$ is Lipschitz in the Wasserstein
distances $W_q$ for all $q>1$.

\subsection{Sobolev Counterparts of the Main Theorems}

For technical reasons, we take a detour to write the above three Theorems by assuming Sobolev regularity (instead of H\"older regularity) of  the observable. This is because, on the one hand, the results with H\"older/Sobolev assumption are almost equivalent thanks to the embedding relations between the two type of functions, and on the other hand, Theorem \ref{Main Theorem: Differentiablility of Bernoulli Convolution with Holder Test Function} is simpler to prove when assuming Sobolev regularity of the observable, while H\"older regularity is more handier for Theorem \ref{Main Thorem: Counter Example due to Dimension Drop}. Theorem \ref{thm:phase transition} is indifferent of the choice, but for completeness we also include its (almost identical) Sobolev version.

\begin{thm}\label{Sobolev counterpart of thm1}
    Let $\phi\in H^s$ and suppose $J\subset(0,1)$ is such that $\mu_\lambda$ has finite $\alpha$-Sobolev norm on $J$. Then assuming $\alpha+2s> 3$, the response function $h(\lambda) = \int_\mathbb{R} \phi(x)\,d\mu_\lambda(x)$ is a.e. differentiable on $J$.
\end{thm}

For the precise definition of Sobolev space $H^s$ and ``finite $\alpha$-Sobolev norm'', see \S. 

\begin{thm}\label{Sobolev counterpart of thm2}
 For any $\lambda_0\in(0,1)$, if $\dim\mu_{\lambda_0}<1$, then for any $s\in(0,\frac{3}{2})$ and for generic $\phi\in H^s$, $h(\lambda) = \int_\mathbb{R} \phi(x)\,d\mu_\lambda(x)$  is not differentiable at $\lambda_0$.
\end{thm}

\begin{thm}\label{Sobolev counterpart of thm3}
    For any $s\in(0,1)$, there exists $\phi\in H^s$ (a smooth cutoff of the Weierstrass function) so that the response function $h(\lambda ) = \int_\mathbb{R} \phi(x)\,d\mu_\lambda(x)$ is a.e. non-differentiable on $(0,\lambda_c)$ and a.e. differentiable on $(\lambda_c,1)$. Here $\lambda_c$ is such that
    $$s+\frac{\log2}{-\log\lambda_c} = 1.$$
    
\end{thm}

We end the introduction by pointing out that, while the above questions pertain to Bernoulli convolutions, it could be interesting to consider similar problems in more general context. A possible route could be:

\begin{que}
Study the regularity of \eqref{Average} in the case where $\mu_\lambda$ is a family of SRB measure for Anosov diffeomorphsims $f_\lambda$ depending smoothly  on a parameter $\lambda.$
\end{que}

We note that for $C^1$ observables the linear response is well understood, see references above.\\

The rest of the paper is organized as follows. 
In \S \ref{sec: Proof of Sobolev LR} we prove Theorem \ref{Sobolev counterpart of thm1} and deduce Theorem \ref{Main Theorem: Differentiablility of Bernoulli Convolution with Holder Test Function} from it.  Theorem~\ref{Main Thorem: Counter Example due to Dimension Drop}  and \ref{Sobolev counterpart of thm2} are proven in \S \ref{Section: Counter Example due to Dimension Drop}. Theorem \ref{thm:phase transition} and \ref{Sobolev counterpart of thm3} are proven in \S\ref{sec: proof of first assertion} and \S\ref{sec: proof of second assertion}. Each section contains separate subsections introducing the lemmas needed for the proof.

\section{Absolute continuity from Sobolev norm}\label{sec: Proof of Sobolev LR}

The purpose of this subsection is to prove Theorem \ref{Sobolev counterpart of thm1}. We begin by the necessary definitions.

Recall the space of Sobolev functions
$$H^s(\mathbb{R}):=\{\phi:\mathbb{R}\to\mathbb{C}:  ||\phi||^2_{H^s} = \int_\mathbb{R} |\hat{\phi}(\xi)|^2 (1+\xi^2)^s\,d\xi<\infty\},$$
and the (uncapped) Sobolev dimension of a finite Borel measure $\mu$ on $\mathbb{R}$:

$$\dim_s (\mu) = \sup\{\alpha\ge0: \int_{\mathbb{R}} |\hat{\mu}(\xi)|^2 (1+|\xi|^2)^{\frac{\alpha-1}{2}} \,d\xi <\infty\}.$$
If $\dim_s \mu\in(0,1)$ then it is equal to the correlation dimension of $\mu$, and if $\dim_s\mu>1$ the $\mu$ is absolutely continuous with respect to the Lebesgue measure and $\frac{d\mu}{dx}\in H^{\dim_s\mu-1}(\mathbb{R})$. For more discussions on Sobolev dimension, see \cite[Section~8.2.1]{d013e7c057ce4b198906d18f73a2536b}.

For our purpose of  studying linear response, the following ``average Sobolev energy'' will be more adapted. Namely, we say that the Bernoulli convolution $\mu_\lambda$ has finite $\alpha$-Sobolev norm on some subinterval $J\subset(0,1)$ if 
\begin{equation}
    \int_J\int_{\mathbb{R}}|\hat{\mu}(\xi)|^2 (1+|\xi|^2)^{\frac{\alpha-1}{2}} \,d\xi\,d\lambda<\infty.
\end{equation}
 Under the convention of Peres-Schlag\cite{peres2000smoothness}, the above is equivalent to $\int_J||\mu_\lambda||_{2,\frac{\alpha-1}{2}}\,d\lambda<\infty$, and they proved that $\dim_s\mu_\lambda \ge \frac{\log2}{-\log\lambda},a.e.\lambda\in (0.5,0.668\cdots)$.

The proof of Theorem \ref{Sobolev counterpart of thm1} is organized as follows: we start with the crucial lemma in \S\ref{sec: drop of Sobolev Norm}  that if $\mu_\lambda$ has finite $\alpha$-Sobolev norm on $J$, then $\frac{d}{d\lambda}\hat{\mu}_\lambda$ has finite $(\alpha-1)$-Sobolev norm on (any subinterval of) $J$.  In \S\ref{sec:concluding Sobolev LR} we conclude the proof using this lemma, and in \S\ref{sec: proof of Holder LR} we cite the result of Peres-Schlag\cite{peres2000smoothness} that affirms the existence of $\gamma>0$ so that $\mu_\lambda$ has finite $\gamma$-Sobolev norm on $(\frac{1}{2}+\epsilon,0.668\cdots)$, which combined with Theorem \ref{Sobolev counterpart of thm1} lead to Theorem \ref{Main Theorem: Differentiablility of Bernoulli Convolution with Holder Test Function}.

\subsection{Drop of Sobolev norm}\label{sec: drop of Sobolev Norm}

This subsection contains the lemma studying the Sobolev norm of the derivative of Bernoulli convolution. Let us begin by several definitions.

Recall that
\[
 X(\lambda)=\sum_{n=0}^{\infty}a_n\lambda^n,
 \qquad \mu_\lambda=\operatorname{Law}(X(\lambda)),\qquad 0<\lambda<1.
\]

The convention for Fourier transform will be $\widehat{\nu}(\xi)=\int e^{i\xi x}\,d\nu(x)$ for any finite Borel measure on $\R$. Define $D_\lambda(\xi):= \partial_\lambda\widehat{\mu_\lambda}(\xi)$. The partial derivative is legitimate, since $\hat{\mu}_\lambda (\xi)= \prod_{k=0}^\infty\cos( \lambda^k\xi)$ is analytic in $\lambda$.

$D_\lambda=\partial_\lambda \mu_\lambda$ in the sense of distributions:
\[
 \langle D_\lambda,\varphi\rangle
: = \frac{d}{d\lambda}\langle\mu_\lambda,\phi\rangle = \frac{d}{d\lambda} \int_{\R} \phi\,d\mu_\lambda =  \E\!\left[\varphi'(X(\lambda))\frac{d}{d\lambda}X(\lambda)\right],
 \qquad \varphi\in C^\infty(\R).
\]

Here the expectation is taken with respect to $a_n , n=0,1,\cdots$. To see why the last equality holds, see the discussion in \S  \ref{Subsubsection: Linear Response with C1 Observable}.

The lemma we need is:

\begin{lm}\label{lemma: Drop of sobolev norm}
    Let $J=(a,b)$ with $0<a<b<1$, and suppose that, for some $\gamma\geq0$,
\begin{equation}\label{eq:assumption}
 A_\gamma(J):=\int_J\int_\R
 (1+\xi^2)^{\gamma}|\widehat{\mu_\lambda}(\xi)|^2\,d\xi\,d\lambda<\infty.
\end{equation}
Then, for every compact interval $K\subset J$,
\begin{equation}\label{eq:conclusion}
 \int_K\int_\R (1+\xi^2)^{\gamma-1}
 |D_\lambda(\xi)|^2\,d\xi\,d\lambda<\infty.
\end{equation}
If $0\leq\gamma<\tfrac12$, one can take $K=J$; in fact,
\begin{equation}\label{eq:global}
 \int_J\int_\R (1+\xi^2)^{\gamma-1}
 |D_\lambda(\xi)|^2\,d\xi\,d\lambda
 \leq C_{J,\gamma}A_\gamma(J).
\end{equation}
\end{lm}

\begin{proof}
Write $F(\lambda,\xi)=\widehat{\mu_\lambda}(\xi)$. We first prove the whole-interval assertion for $\gamma<\tfrac12$. Differentiating the convergent product $F(\lambda,\xi)=\prod_{k=0}^{\infty}\cos(\lambda^k\xi)$ gives
\[
 \partial_\lambda F(\lambda,\xi)
 =-\xi\sum_{n=1}^{\infty}n\lambda^{n-1}\sin(\lambda^n\xi)
   \prod_{\substack{k\geq0\\k\neq n}}\cos(\lambda^k\xi).
\]
Since every sine and cosine has modulus at most one,
\[
 \left|\sin(\lambda^n\xi)
   \prod_{\substack{k\geq0\\k\neq n}}\cos(\lambda^k\xi)\right|
 \leq\prod_{k=n+1}^{\infty}|\cos(\lambda^k\xi)|
 =|F(\lambda,\lambda^{n+1}\xi)|.
\]
Set $M_\gamma(\lambda):=\||\xi|^\gamma F(\lambda,\xi)\|_{L^2(d\xi)}$. Minkowski's inequality and the substitution $\eta=\lambda^{n+1}\xi$ yield
\begin{align*}
 \bigl\||\xi|^{\gamma-1}D_\lambda(\xi)\bigr\|_{L^2(d\xi)}
 &\leq \sum_{n=1}^{\infty}n\lambda^{n-1}
       \bigl\||\xi|^\gamma F(\lambda,\lambda^{n+1}\xi)\bigr\|_{L^2(d\xi)}\\
 &=\lambda^{-\gamma-3/2}
   \sum_{n=1}^{\infty}n\lambda^{n(1/2-\gamma)}M_\gamma(\lambda)\\
 &=\frac{\lambda^{-1-2\gamma}}
        {(1-\lambda^{1/2-\gamma})^2}\,M_\gamma(\lambda).
\end{align*}
The series converges for $\gamma<\tfrac12$, so the calculation is justified by first using finite sums and then taking their weighted $L^2$ limit. Squaring and integrating over $\lambda\in J$ proves \eqref{eq:global}.

Next we deal with the case of arbitrary $\gamma\geq0$. We will apply Fourier transform in the $\lambda$ variable and apply stationary phase estimate. Put $\br{\xi}=(1+\xi^2)^{1/2}$. Fix $\chi\in C_c^\infty(J)$ equal to one near $K$, 
\begin{align*}
    \int_K\int_{\R} \langle\xi\rangle^{2\gamma-2} |D_\lambda(\xi)|^2\,d\xi\,d\lambda &\le \int_{\R} \langle\xi\rangle^{2\gamma-2}\int_{\R} \chi^2(\lambda) |\frac{d}{d\lambda} F(\lambda,\xi)|^2\,d\lambda\,d\xi\\
    &\approx\int_{\R} \langle\xi\rangle^{2\gamma-2}\int_{\R} | \mathcal{F}_\lambda \left(\chi(\lambda) \frac{d}{d\lambda} F(\lambda,\xi) \right)(\eta)|^2\,d\eta\,d\xi\\
\end{align*}
where $\mathcal{F}_\lambda$ stands for Fourier transform with respect to $\lambda$, and $\approx$ means equality holds up to a multiple of some absolute constant (coming from Plancherel, for example). Moreover, using integral by parts, the inner integrand is 
\begin{align*}
     \mathcal{F}_\lambda \left(\chi(\lambda) \frac{d}{d\lambda} F(\lambda,\xi) \right)(\eta)&= \int_{\R} \chi(\lambda) \frac{d}{d\lambda}F(\lambda,\xi) e^{-i\eta\lambda}\,d\lambda\\
     &=-\int_{\R}e^{-i\eta\lambda} (\chi'(\lambda) - i\eta \chi(\lambda))F(\lambda,\xi)\,d\lambda\\
     &= -\mathcal{F}_\lambda (\chi'F(\cdot,\xi))(\eta) + i\eta \mathcal{F}_\lambda(\chi F(\cdot,\xi))(\eta)
\end{align*}
and thus 
\begin{align*}
      || \mathcal{F}_\lambda \left(\chi \frac{d}{d\lambda} F(\cdot,\xi) \right)(\eta)||^2_{L^2(d\eta)} &\le (\left || \mathcal{F}_\lambda \left(\chi'  F(\cdot,\xi) \right)(\eta)||_{L^2(d\eta)} +  || \eta \mathcal{F}_\lambda \left(\chi F(\cdot,\xi) \right)(\eta)||_{L^2(d\eta)}\right)^2\\
      &\le 2||\chi'(\lambda) F(\lambda,\xi)||^2_{L^2(d\lambda)} +  2|| \eta \mathcal{F}_\lambda \left(\chi F(\cdot,\xi) \right)(\eta)||^2_{L^2(d\eta)}\\
      &\le C|| F(\lambda,\xi)||^2_{L^2(d\lambda)} +  2|| \eta \mathcal{F}_\lambda \left(\chi F(\cdot,\xi) \right)(\eta)||^2_{L^2(d\eta)}.
\end{align*}
Here $C = 2||\chi||_{C^1}$. Consequently, 
\begin{align*}
        \int_K\int_{\R} \langle\xi\rangle^{2\gamma-2} |D_\lambda(\xi)|^2\,d\xi\,d\lambda&\le C\int_K \langle \xi \rangle^{2\gamma-2} || F(\lambda,\xi)||^2_{L^2(d\lambda)}\,d\xi \\
        &+  2\int_K \langle \xi \rangle^{2\gamma-2} || \eta \mathcal{F}_\lambda \left(\chi F(\cdot,\xi) \right)(\eta)||^2_{L^2(d\eta)}\,d\xi,
\end{align*}
and the first term above causes no harm, as it is no more than $A_\gamma(J)<\infty$.  Thus we need to study the $L^2(d\eta)$ norm of
\begin{align*}
     G(\eta,\xi)& :=\int_\R e^{-i\eta\lambda}\chi(\lambda)F(\lambda,\xi)\,d\lambda\\
     &=\int_{\R} e^{-i\eta\lambda }\chi(\lambda) \E[e^{i\xi X(\lambda)}] \\
 &=\E\int_\R\chi(\lambda)e^{i(\xi X(\lambda)-\eta\lambda)}\,d\lambda.
\end{align*}
We split into twp cases depending on whether $|\eta|$ is large or small compared to $\xi$. Namely, for a sufficiently large $M$ depending only on $J$, the uniform bound on $\frac{d}{d\lambda}X(\lambda)$ gives
\[
 |\xi \frac{d}{d\lambda}X(\lambda)-\eta|\geq\tfrac12|\eta|
 \qquad\text{if }|\eta|\geq M\br{\xi}.
\]
Repeated integration by parts, using
\[
 e^{i(\xi X(\lambda)-\eta\lambda)}
 =\frac{1}{i(\xi X(\lambda)'-\eta)}
   \partial_\lambda e^{i(\xi X(\lambda)-\eta\lambda)},
\]
therefore gives, for every integer $N\geq1$,
\begin{equation}\label{eq:nonstationary}
 |G(\eta,\xi)|\leq C_N(1+|\eta|)^{-N},
 \qquad |\eta|\geq M\br{\xi}.
\end{equation}
since the compact support of $\chi$ eliminates boundary terms.

On $|\eta|\leq M\br{\xi}$, we simply have $\eta^2\br{\xi}^{2\gamma-2}\leq M^2\br{\xi}^{2\gamma}$. Therefore
\begin{align*}
   & \int_K \langle \xi \rangle^{2\gamma-2}|| \eta \mathcal{F}_\lambda \left(\chi F(\cdot,\xi) \right)(\eta)||^2_{L^2(d\eta)}\,d\xi\\
    &= \int_K \langle \xi \rangle^{2\gamma-2}\left(   \int_{|\eta|<M \langle\xi\rangle}   \eta^2 G(\eta,\xi)^2 \,d\eta +\int_{|\eta|\ge M \langle\xi\rangle}   \eta^2 G(\eta,\xi)^2 \,d\eta \right)  \,d\xi\\
    &\le M^2 \int_K \langle \xi \rangle^{2\gamma-2}   \langle\xi\rangle^2 ||G(\eta,\xi)||^2_{L^2(d\eta)}\,d\xi + C_N \int_K \langle \xi \rangle^{2\gamma-2}  \int_{|\eta|\ge M\langle\xi\rangle} (1+|\eta|)^{-N+2}\,d\eta  \,d\xi\\
    &\approx \int_K \int_\R \langle\xi\rangle^{2\gamma}  F(\lambda;\xi)^2\,d\lambda\,d\xi+ 1\\
    &\le A_\gamma(J)+1<\infty.
\end{align*}
This proves \eqref{eq:conclusion}.
\end{proof}

\subsection{Proof of Theorem \ref{Sobolev counterpart of thm1}}\label{sec:concluding Sobolev LR}

In this section we prove Theorem \ref{Sobolev counterpart of thm1}, by showing that $h(\lambda) = \int\phi\,d\mu_\lambda$ is absolutely continuous. Without loss of generality, assume $J$ has non-empty interior. For Lebesgue almost every $a,b\in J$,
\begin{equation}\label{eq: Fourier transform of Difference}
      h(b)-h(a) = \frac{1}{2\pi} \int_{\mathbb{R}} \hat{\phi}(\xi) (\hat{\mu}_{b}(\xi) - \hat{\mu}_a (\xi) )\,d\xi = \frac{1}{2\pi}\int_{\mathbb{R}} \hat{\phi}(\xi) \int_a^b \frac{d\hat{\mu}_{\lambda}(\xi)}{d\lambda}\,d\lambda \,d\xi,
\end{equation}
where $\hat{\phi}(\xi):= \int_{\mathbb{R}} e^{- i \xi x} \phi(x)\,dx$ and $\hat{\mu}_{\lambda} (\xi) = \int_{\mathbb{R}} e^{ i \xi x}\, d\mu_{\lambda}(x)$.  The first equality holds since  Lebesgue almost every $b\in J$,
\begin{equation}\label{eq:Fourier transform of Response function}
    h(b) = \int_\R h(x)\;d\mu_\lambda(x) = \frac{1}{2\pi}\int_{\R} \hat{h}(\xi) \hat{\mu}_b(\xi)\,d\xi,
\end{equation}

which we will justify at the end. By Equation \ref{eq: Fourier transform of Difference}, 
\begin{equation}
    |h(b) - h(a)| \le \frac{1}{2\pi} \int_a^b \int_\R |\hat{\phi}(\xi) D_\lambda(\xi)|\,d\xi\,d\lambda.
\end{equation}
Thus to prove absolutely continuity of $h$, we only need to show that for any compact subinterval $K\Subset J$, 
\begin{equation}
    \int_K\int_\R |\hat{\phi}(\xi) D_\lambda(\xi)|\,d\xi\,d\lambda<\infty.
\end{equation}

Let $q\in\R$ be such that 
$$q<s , \quad  1-q<\frac{\alpha-1}{2}.$$
Such a choice is possible since $2s+\alpha>3$. By Cauchy-Schwartz applied twice,
\begin{align}
  \nonumber    \int_K\int_\R |\hat{\phi}(\xi) D_\lambda(\xi)|\,d\xi\,d\lambda&\le   \int_K \sqrt{\int_\R |\hat{\phi}(\xi)|^2 (1+\xi^2)^{q} \,d\xi \int_\R |D_\lambda(\xi)|^2 (1+\xi^2)^{-q}\,d\xi}\,\,d\lambda\\
\nonumber      &\le ||\phi||_{H^s} \sqrt{|K| \int_K \int_\R |D_\lambda(\xi)|^2 (1+\xi^2)^{-q}\,d\xi\,d\lambda}\\
      &\le C ||\phi||_{H^s} (1+A_{-q+1}(J))<\infty.\label{eq:finite formal derivative}
\end{align}
Here the last inequality follows by Lemma \ref{lemma: Drop of sobolev norm}. 

Finally, we prove \eqref{eq:Fourier transform of Response function}. By Fubini, it holds as long as 
$$\int_{\R} |\hat{h}(\xi) \hat{\mu}_b(\xi)|\,d\xi<\infty.$$
Since $h\in H^s$ , $\int_J\int_\R
 (1+\xi^2)^{\frac{\alpha-1}{2}-\epsilon}|\widehat{\mu_\lambda}(\xi)|^2\,d\xi\,d\lambda<\infty$ and $2s+\alpha>3$, one can prove the above for a.e. $b\in J$ using the same argument as in \eqref{eq:finite formal derivative}.  The theorem is proved.\qed

\subsection{Deducing Theorem \ref{Main Theorem: Differentiablility of Bernoulli Convolution with Holder Test Function}} \label{sec: proof of Holder LR}

In this section we prove Theorem \ref{Main Theorem: Differentiablility of Bernoulli Convolution with Holder Test Function}, using the above discussion as well as the following result by Peres and Schlag:

\begin{lm}[Lemma $5.5$ of \cite{peres2000smoothness}]\label{Lemma: Finite Fractional Sobolev Norm}
    For any $\epsilon>0$, there exists a $\gamma = \gamma(\epsilon)>0$ so that, 
    $$\int_{\frac{1}{2}+\epsilon}^{\frac{1}{\sqrt{2}}} \int_\R |\hat{\mu}_\lambda (\xi)|^2(1+\xi^2)^\gamma \,d\lambda<\infty.$$
\end{lm}

That is, $\mu_\lambda$ has finite $(1+2\gamma)$-Sobolev norm on $I_0:=(\frac{1}{2}+\epsilon , 2^{-\frac{1}{2}})$, which in turn implies that $\mu_\lambda$ has finite $(1+2\gamma)$-Sobolev norm on any compact subinterval $K\Subset(2^{-\frac{1}{2}}+\epsilon,1)$. Indeed, fix any $K = [u,v]$ with $u>2^{-\frac{1}{2}}+\epsilon$. $K$ can be covered by intervals of the form $I_m :=((\frac{1}{2}+\epsilon)^{\frac{1}{m}} , 2^{-\frac{1}{2m}})$ by a simple calculation.    Moreover, the following holds thanks to the infinite product structure of $\hat{\mu}_\lambda(\xi)$:
\begin{equation}\label{eq:m-fold convolution}
    \hat{\mu}_{\lambda}(\xi) = \hat{\mu}_{\lambda^m}(\xi) \hat{\mu}_{\lambda^m}(\lambda\xi)\cdots \hat{\mu}_{\lambda^m}(\lambda^{m-1}\xi). 
\end{equation}
In particular, $|\hat{\mu}_\lambda(\xi)|\le |\hat{\mu}_{\lambda^m}(\xi)|,\forall m\in\mathbb{Z}_{\ge 1}$, and a simple change of variable shows, for any $\gamma\in\R$,
$$\int_{I_0}\int_\R (\xi^2+1)^\gamma |\hat{\mu}_\lambda(\xi)^2|\,d\xi\,d\lambda<\infty \Rightarrow  \int_{I_m}\int_\R (\xi^2+1)^\gamma |\hat{\mu}_\lambda(\xi)^2|\,d\xi\,d\lambda<\infty.$$
And since $K$ is covered by finitely $I_m$, $\int_{K}\int_\R (\xi^2+1)^\gamma |\hat{\mu}_\lambda(\xi)^2|\,d\xi\,d\lambda<\infty$ as well. 

In conclusion, Lemma\ref{Lemma: Finite Fractional Sobolev Norm} implies, for any $\epsilon>0$ there exists $\gamma = \gamma(\epsilon)>0$ such that $\mu_\lambda$ has finite $(1+2\gamma)$-Sobolev norm on any compact subinterval $K\Subset(2^{-\frac{1}{2}}+\epsilon,1)$. Thus by Theorem \ref{Sobolev counterpart of thm1}, if $s$ is such that $s>1-\gamma$ and $\phi\in H^s$, then $h(\lambda) = \int_\R h(x)\,d\mu_\lambda(x)$ is a.e. differentiable on $K$, and thus on $(2^{-\frac{1}{2}}+\epsilon,1)$. Now we only need to recall that for any $s\in(0,1)$ and any $\alpha>s$, any compactly supported function $\phi\in C^\alpha$ is also in $H^s$ \cite[Section~1.3.4]{bahouri2011fourier}, and  Theorem \ref{Main Theorem: Differentiablility of Bernoulli Convolution with Holder Test Function} follows.
\section{Nondifferentiability due to Dimension Drop}\label{Section: Counter Example due to Dimension Drop}

\subsection{ Proof strategy.}
In this section we prove Theorem \ref{Main Thorem: Counter Example due to Dimension Drop}. The proof consists of the following two lemmas:

\begin{lm}\label{Lemma: Existence of a blowup test function}
    Under the assumptions of Theorem \ref{Main Thorem: Counter Example due to Dimension Drop}, there exists $\phi\in C^\theta$ such that 
    $$\liminf_{\epsilon\to 0} \frac{h_\phi(\lambda_0+\epsilon) - h_\phi(\lambda_0)}{\epsilon} = \infty.$$
\end{lm}

\begin{lm}\label{Lemma: One blowup implies genericity}
    Under the assumptions
of Theorem \ref{Main Thorem: Counter Example due to Dimension Drop},
   if there exists $\phi \in C^\theta$ and $\{\epsilon_n\}_{n\in \mathbb{N}}$ converging to $0$ such that 
   $\DS \lim_{n\to \infty} \left|\frac{h_\phi(\lambda_0+\epsilon_n) - h_\phi(\lambda_0)}{\epsilon_n}\right| = \infty,$ then  $h_{\psi}$ is not differentiable at $\lambda_0$ for generic $\psi\in C^\theta$. 
\end{lm}

Clearly the above two lemmas imply Theorem \ref{Main Thorem: Counter Example due to Dimension Drop}.

The proof of Lemma \ref{Lemma: One blowup implies genericity} is elementary but of independent interest, as it easily generalized to any one parameter family of measures. As for Lemma \ref{Lemma: Existence of a blowup test function}, the proof strategy is as follows:

We will take advantage of the fact that if $\mu_{\lambda_0}$ has dimension less than $1$, then there are sets with small Lebesgue measure but large $\mu_{\lambda_0}$-measure. One can thus arrange a sequence of  $C^1$ test functions $\phi_n$ so that it concentrate more and more on these sets, that they converge to a $C^\theta$ function $\phi$, and that the derivative of $\lambda\mapsto \int\phi\,d\mu_{\lambda}$ blows up to infinity. 

More precisely, we will reduce the proof to the existence of a family of special test functions $\{\phi_n\}_{n\in\mathbb{N}}$, as in the following proposition

\begin{prop}\label{Proposition: Non-diff example; Intermediate Step}
   Under the assumption of Theorem \ref{Main Thorem: Counter Example due to Dimension Drop}, there exists a family of $C^1$, nonnegative functions $\{\phi_n\}$ such that:
    \begin{enumerate}[(1)]
        \item There exists $\epsilon_0 >0$ such that for all $n$ and all $\epsilon$ with $|\epsilon|<\epsilon_0$, $h_n'(\lambda_0+\epsilon)$ exist and are nonnegative. Here $ h_n(\lambda) = h_{\phi_n}(\lambda):=\int \phi_n \,d\mu_{\lambda}$.
        \item $\DS \sum_{n} h_n'(\lambda_0) = \infty$.
        \item  $\DS \phi:=  \sum  \phi_n$ is bounded and H$\ddot{o}$lder continuous, with H$\ddot{o}$lder exponent $\theta$.
    \end{enumerate}
\end{prop}

Note that it is part of the proposition that $h_n$'s are differentiable. This is because $\phi_n$'s are $C^1$ as will be explained in \S \ref{Subsection: Proof of Main Proposition}.

Let us also point out that, once Theorem \ref{Main Thorem: Counter Example due to Dimension Drop} is proved, it is easy to deduce Theorem \ref{Sobolev counterpart of thm2}, since for any $s\in(0,1)$ and any $\alpha>s$, any compactly supported function $\phi\in C^\alpha$ is also in $H^s$ \cite[Section~1.3.4]{bahouri2011fourier}. Strictly speaking, we didn't assume $\phi$ to be compactly supported, but since $\mu_\lambda,\lambda\in K$ are supported on a common compact set (depending only on $K$), we may multiply $\phi$ by a smooth cutoff function and assume it is compactly supported. 

In the rest of this section, we first show how to prove Theorem \ref{Main Thorem: Counter Example due to Dimension Drop} assuming Proposition  \ref{Proposition: Non-diff example; Intermediate Step} in \S \ref{Subsectioin: proof of counterexample assuming proposition}. Then we go to construction of such family $\{\phi_n\}$ in \S \ref{Subsection: Proof of Main Proposition}. Finally, we present the proof of Lemma \ref{Lemma: One blowup implies genericity} in \S \ref{Subsection: proof of lemma: one blow up implies genericity}.   

We end this subsection by recalling the notion the Hausdorff dimension of subsets of $\mathbb{R}$ \cite{O’Neil_1999} . 

\begin{df}
   Fix any subset $A\subset\mathbb{R}$. For any $s>0,\delta>0$, define 
    $$H_\delta^s(A):= \inf\{\sum_j |I_j|^s: \{I_j\} \text{ is a cover of } A \text{ s.t. the diameter }|I_j|< \delta ,\forall j\}.$$

Here the diameter $|X|$ of any subset of $\mathbb{R}$ is defined as $\DS |X|:=\sup_{x,y\in X}|x-y|$.

By construction for fixed $s$, $H^s_\delta(A)$ is increasing in $\delta$, so the limit 
$\DS H^s(A):= \lim_{\delta\to0^+} H^s_\delta(A)$ is well defined (but may equal infinity). Moreover, there exists a unique $s_0\ge 0$ such that
$$s_0 = \inf\{s\ge 0: H^s(A)=0\} = \sup\{s\ge 0: H^s(A)=\infty\}.$$
The Hausdorff dimension of $A$ is then defined to be $\dim_H A = s_0$.
\end{df}

\subsection{Proof of Theorem \ref{Main Thorem: Counter Example due to Dimension Drop} using Proposition \ref{Proposition: Non-diff example; Intermediate Step}} \label{Subsectioin: proof of counterexample assuming proposition}

\begin{proof}
    Given $\{\phi_n\}$ and $\phi$, define $h(\lambda)  = \int \phi\,d\mu_{\lambda}, \lambda\in(\lambda_0-\epsilon_0,\lambda_0+\epsilon_0)$ . We will show that $\DS \liminf_{\epsilon\to 0} \frac{h(\lambda_0+\epsilon) - h(\lambda_0)}{\epsilon} = \infty$, so that $h$ is not differentiable at $\lambda_0$. 

Let $ F(\epsilon,n) := \frac{h_n(\lambda_0+\epsilon) - h_n(\lambda_0)}{\epsilon}, n\in\mathbb{N},0<|\epsilon|<\epsilon_0$. Note that 
$$\lim_{\epsilon\to 0} F(\epsilon,n) = h_n'(0),\quad \forall n\in\mathbb{N}$$
and 
$\DS F(\epsilon,n)\ge 0,\quad  \forall 0<|\epsilon|<\epsilon_0,  n\in\mathbb{N}$
since $h_n'(\lambda_0+\epsilon)\ge 0$ for all $0<|\epsilon|<\epsilon_0$. 
Thus 

\begin{align}
    \liminf_{\epsilon\to 0} \frac{h(\lambda_0+\epsilon) - h(\lambda_0)}{\epsilon} &= \liminf_{\epsilon
    \to 0} \frac{\int \sum_{n=1}^\infty \phi_n \,d\mu_{\lambda_0+\epsilon} - \int\sum_{n=1}^\infty \phi_n \,d\mu_{\lambda_0} }{\epsilon}\nonumber \\
    &= \liminf_{\epsilon
    \to 0}  \frac{\sum_{n=1}^\infty  (\int  \phi_n \,d\mu_{\lambda_0+\epsilon} - \int \phi_n \,d\mu_{\lambda_0}) }{\epsilon} \nonumber\\
    &= \liminf_{\epsilon\to 0} \int F(\epsilon,n) \,dc(n) \quad c \text{ is the counting measure on }\mathbb{N}\nonumber\\
    &\ge \int\liminf_{\epsilon\to 0} F(\epsilon,n) \,dc(n) \quad \text{ by Fatou's Lemma }\nonumber \\
    &= \sum_{n=1}^\infty h_n'(0) = \infty,
\end{align}
which implies $h$ is not differentiable at $0$. Note that the second equality holds because 
$\DS \sum_n \int \phi_n\,d\mu_{\lambda_0+\epsilon}$ is absolutely convergent. 
\end{proof}

\subsection{Proof of Proposition \ref{Proposition: Non-diff example; Intermediate Step}}\label{Subsection: Proof of Main Proposition}
In this subsection we prove Proposition \ref{Proposition: Non-diff example; Intermediate Step}. We start by deriving a formula of $h'_\phi(\lambda)$ if $\phi$ is $C^1$ in \S \ref{Subsubsection: Linear Response with C1 Observable}, which will allow us to manipulate the term $\DS h_{\phi_n}'(\lambda)$ in Proposition \ref{Proposition: Non-diff example; Intermediate Step}. We will then proceed to finish the proof.

\subsubsection{Linear Response for $C^1$ Observables}\label{Subsubsection: Linear Response with C1 Observable}

Recall that $\mu_\lambda$ is the law of
     $\DS X(\lambda)\!\!=\!\! \sum_{m=1}^\infty a_m \lambda^{m-1} $, $\lambda\in(0,1)$
    where the $a_m$'s are i.i.d. random variables taking values in $\{-1,1\}$ with probability $\{\frac{1}{2},\frac{1}{2}\}$. In particular, for any fixed $a_i$'s, the right hand side is analytic in $\lambda$. To emphasize the dependence of $X(\lambda)$ on $a_i$'s, we write it as $\DS X(\lambda) = X(\lambda;a)$, where $a = (a_1,a_2,\cdots)\in\Sigma:=\{-1,1\}^{\mathbb{N}}$.
     
     Let $\nu$ be the $(\frac{1}{2},\frac{1}{2})$-Bernoulli measure on the symbolic space $\Sigma$. Then by definition, $\DS \mu_\lambda = \left(X(\lambda;\cdot)\right)_{*}\nu$, and therefore
     $$h_\phi(\lambda) = \int_\mathbb{R} \phi(x)\,d\mu_{\lambda}(x) = \int_{\Sigma}\phi(X(\lambda;a)) \,d\nu(a).$$

     This enables us to write down an explicit formula for the derivative if $\phi$ is $C^1$:
    \begin{lm}\label{Lemma: Linear Response for C1 observables}
        Suppose $\phi\in C^1$, then for any $\lambda_0\in (0,1)$, $h_\phi$ is differentiable at $\lambda_0$ with
        $$h'_\phi(\lambda_0) = \int_\Sigma  \phi'(X(\lambda_0;a)) X^{(1)}(\lambda_0;a)\,d\nu(a). $$

        Here $\DS X^{(1)}(\lambda;a):= \sum_{m\ge 1} m a_{m+1}\lambda^{m-1}$ is the formal derivative $\frac{\partial X(\lambda;a)}{\partial\lambda}$.
    \end{lm}

\begin{proof}
We wish to switch the order of derivative and integration, so that
\begin{equation}\label{Equation: Switch the order of derivative and integration}
    \frac{d}{d\lambda}\left(\int_{\Sigma}\phi(X(\lambda;a)) \,d\nu(a)\right)|_{\lambda=\lambda_0} = \int_{\Sigma} \frac{d}{d\lambda} \left( \phi(X(\lambda;a))  \right)|_{\lambda=\lambda_0} \,d\nu(a).
\end{equation}  

    The lemma follows by expanding the right hand side by chain rule.

   To justify the above equality, note that $X(\cdot;\cdot)$ is uniformly bounded for all $a\in \Sigma$ and $\lambda$ near $\lambda_0$. Since $\phi'$ is continuous, it is also uniformly bounded on the domain of $X$. Moreover, $X^{(1)}$ is also uniformly bounded for all $a\in \Sigma$ and $\lambda$ near $\lambda_0$. It follows that
    $$\int_\Sigma  \phi'(X(\lambda;a)) X^{(1)}(\lambda;a)\,d\nu(a)$$
    is uniformly bounded for all $a\in \Sigma$ and $\lambda$ near $\lambda_0$. Thus equation \eqref{Equation: Switch the order of derivative and integration} is valid.
\end{proof}

We include two more technical lemmas needed later. The first is the formula for $h_\phi''(\lambda)$ provided $\phi$ is $C^2$:

\begin{lm}\label{Lemma: Linear response for C2 Observables}
      Suppose $\phi\in C^2$, then for any $\lambda_0\in (0,1)$, $h_\phi$ is twice differentiable at $\lambda_0$ with
        $$h''_\phi(\lambda_0) = \int_\Sigma  \phi''(X(\lambda_0;a)) \left(X^{(1)}(\lambda_0;a)\right)^2 + \phi'(X(\lambda_0;a))X^{(2)}(\lambda_0;a) \,d\nu(a). $$

        Here $\DS X^{(2)}(\lambda;a):= \sum_{m\ge 2} m(m-1) a_{m+1}\lambda^{m-2}$ is the formal derivative $\frac{\partial^2 X(\lambda;a)}{\partial\lambda^2}$.
\end{lm}

The proof is completely analogous to that of Lemma \ref{Lemma: Linear Response for C1 observables} and we omit it.

The second lemma says that $X^{(1)}$ is not too small if $X$ itself is large:

\begin{lm}\label{Lemma: X^(1) is not too small if X is large}
    For any $\delta\in(0,\frac{1}{2})$, there exists $\delta' = \delta'(\delta)>0$ satisfying the following. For any $\lambda\in(\delta,1-\delta)$ and $a\in \Sigma$, if $X(\lambda;a)>\frac{1}{1-\lambda}-\delta'$, then $X^{(1)}(\lambda;a)>\frac{1}{2}$.
\end{lm}

\begin{proof}
    Fix an $m\in\mathbb{N}$ such that $m\ge 10$ and that 
    $$\sum_{n=m}^\infty n(1-\delta)^{n-1}<\frac{1}{2}.$$

    Clearly $m$ only depends only on $\delta$. Let $\delta' = \frac{\delta^{m-1}}{2}$. We claim that the lemma holds with this choice of $\delta'$. 

    Suppose $X(\lambda;a)>\frac{1}{1-\lambda} -\delta'$. Observe that $a$ must satisfy $a_1 = a_2 = \cdots = a_{m}=1$. Otherwise
    \begin{align*}
       \frac{\delta^{m-1}}{2}=\delta'> \frac{1}{1-\lambda} - X(\lambda;a) 
       &  =\sum_{n=0}^\infty \lambda^n - \sum_{n=0}^\infty a_{n+1}\lambda^n\\
        =\sum_{n=0}^{m-1} (1-a_{n+1})\lambda^n  + \sum_{n=m}^\infty  (1-a_{n+1})\lambda^n
        &\ge\lambda^{m-1}>\delta^{m-1}
    \end{align*}
    a contradiction. It follows that 
    \begin{align*}
        X^{(1)}(\lambda;a) = \sum_{n=1}^\infty na_{n+1}\lambda^{n-1} 
        &= \sum_{n=1}^{m-1} na_{n+1}\lambda^{n-1} + \sum_{n=m}^{\infty} na_{n+1}\lambda^{n-1}\\
        \ge 1-\sum_{n=m}^{\infty} n\lambda^{n-1}
        &>1-\sum_{n=m}^\infty n(1-\delta)^{n-1}>\frac{1}{2}
    \end{align*}
    as desired.
\end{proof}
    
We are now ready to prove Proposition \ref{Proposition: Non-diff example; Intermediate Step}.

\begin{proof}
   For simplicity, we denote $\mu_{\lambda_0}$ by $\mu$. Suppose $\dim\mu = \alpha<1$ and fix $\beta\in\mathbb{R}$ with  $\alpha<\beta<1$. Given $\theta\in(0,1)$,  we wish to construct a H\"older continuous function with H\"older exponent $\theta$ so that $\lambda\mapsto\int\phi \,d\mu_{\lambda}$ is not differentiable at $\lambda_0$.

    By definition, there exists $A\subset \mathbb{R}$ so that $\mu(A)=1$ and $\dim_H (A) < \beta$. Without loss of generality, assume $A$ is contained in supp$\mu = [-\frac{1}{1-\lambda_0},\frac{1}{1-\lambda_0}]$.

    Let $f_1(x) = \lambda_0 x-1,f_2(x) = \lambda_0x+1$. By self-similarity of $\mu$, it is not hard to see that for any $m\in\mathbb{N}$,  
    $$\mu\left(f_2^m(A)\right) = 2^{-m}$$
    where $f_2^m$ denotes $m$ copies of $f_2$ composed together.

    By Lemma \ref{Lemma: X^(1) is not too small if X is large}, there exists $\epsilon_0>0$ and $\delta'>0$ such that
    \begin{equation}\label{equation: X^(1) is not too small if X is large}
        \forall \lambda\in (\lambda_0-\epsilon_0,\lambda_0+\epsilon_0)\subset(0,1) ,\quad X(\lambda;a)>\frac{1}{1-\lambda} - 4\delta' \Rightarrow X^{(1)}(\lambda;a)>\frac{1}{2}.
    \end{equation}
    Moreover, it is clear that once such a pair $\epsilon_0,\delta'$ is chosen, one can further shrink $\epsilon_0$ to be arbitrarily small so that \eqref{equation: X^(1) is not too small if X is large} remains valid. It follows that one can fix an $m_0\in\mathbb{N}$ such that  
    $$\delta'> \frac{2\lambda_0^{m_0}}{1-\lambda_0}+ \frac{1}{1-\lambda_0-\epsilon_0}-\frac{1}{1-\lambda_0}.$$
    
    Define  $A_0 = f_2^{m_0} (A)$. Since $f_2$ is bi-Lipschitz, $\dim_H(A_0) = \dim_H(A)<\beta$. A straightforward computation shows that 
    \begin{equation}\label{equation:lower bound for A_0}
        \left[\frac{1}{1-\lambda_0-\epsilon_0}-\delta',\frac{1}{1-\lambda_0}\right]\subset A_0 \subset 
        \left[\frac{1}{1-\lambda_0} - \frac{2\lambda_0^{m_0}}{1-\lambda_0} , \frac{1}{1-\lambda_0}\right].
    \end{equation}

    By the definition of Hausdorff dimension, for any $n\in\mathbb{N}$, there exist $\delta_0 = \delta_0(n)>0$ so that for any $\delta\in (0,\delta_0)$, one can find a family of closed intervals $\{I_j\}$ satisfying \vskip2mm
  
        (1)     $\{I_j\}$ is an $\delta$-cover of $A_0$. That is,  $\DS A_0\subset \cup I_j$ and $|I_j|<\delta,\forall j$. Here $|I_j|$ stands for the diameter of $I_j$, which equals to its Lebesgue measure $m(I_j)$, since $I_j$'s are intervals. Moreover, by shrinking these intervals one can always assume that each $I_j$ is contained in the $\delta'$neighborhood of $A_0$: $\DS \sup\{ |x-y|: x\in A_0 , y\in  \bigcup_j I_j  \}<\delta'$.
        
         (2) $\DS \sum_{j=1}^\infty |I_j|^\beta<\frac{1}{n}.$\vskip2mm
        
Fix an $0<\delta<\delta_0(n)$ with
\begin{equation}
    \delta^{1-\beta} <  n^{-1-\frac{2\theta}{1-\theta}}
\end{equation}
and let $\{I_j\}$ be a family of closed intervals satisfying the above requirements. In particular, $\mu(\cup_{j=1}^{\infty} I_j) \ge \mu(A_0)= 2^{-m_0}=:C$. Pick a finite subfamily $\{I_j\}_{j=1}^N$ so that 
\begin{equation}\label{Condition on the subfamily intervals}
    \mu\left(\bigcup_{j=1}^{N} I_j\right) >\frac{3}{4}C.
\end{equation}
 Let $\phi_n$ be a $C^1$ function with the following properties.
\begin{enumerate}
    \item $ \phi(-\infty)=0 $
    \item $0\le \phi_n'\le 1$ and  $ \phi_n'=1$ on $\DS \bigcup_{j=1}^N I_j$.
    \item $\phi_n'$ is zero outside the $\delta'$-neighborhood of $\DS  \bigcup_{j=1}^N I_j$
    and $\int_{\mathbb{R}} \phi_n' (x)\,dx <2m(\cup_{j=1}^N I_j)$. Here $m$ is the (normalized) Lebesgue measure.
\end{enumerate}

Note that 
\begin{align}
    m\left(\bigcup_{j=1}^N I_j\right) \le \sum_{j=1}^N m(I_j) < \delta^{1-\beta} \sum_{j=1}^N |I_j|^{\beta} < \delta^{1-\beta} \sum_{j=1}^{\infty} |I_j|^{\beta} < n^{-\frac{2\theta}{1-\theta}} n^{-2}
\end{align}
by our choice of $\delta$. Thus
\begin{equation}\label{Sup norm Bound of phi_n}
    ||\phi_n||_{C^0} <2n^{-\frac{2\theta}{1-\theta}} n^{-2}.
\end{equation}

Next we would like to bound the $\theta$-H$\ddot{o}$lder norm of $\phi_n$. Define $\rho = \rho_n = n^{-\frac{2}{1-\theta}}$. Suppose $x\neq y \in \mathbb{R}$ are such that $|x-y|<\rho$, then
\begin{equation}
    \frac{|\phi_n(x)-\phi_n(y)|}{|x-y|^\theta} \le \frac{||\phi_n'||_{C^0} |x-y|}{|x-y|^\theta} \le |x-y|^{1-\theta} \le n^{-2}.
\end{equation}
    If $|x-y|\ge \rho$, then 
    \begin{equation}
           \frac{|\phi_n(x)-\phi_n(y)|}{|x-y|^\theta} \le  2\rho^{-\theta}  ||\phi_n||_{C^0} \le4 n^{\frac{2\theta}{1-\theta}} n^{-\frac{2\theta}{1-\theta}} n^{-2} = 4n^{-2}.
    \end{equation}

    Combining the two inequalities above we conclude 
    \begin{equation}\label{Bound on Holder Norm of phi_n}
         ||\phi_n||_{C^\theta} \le 4n^{-2}.
    \end{equation}
     We claim that $\{\phi_n\}$ satisfies the assumptions in Proposition \ref{Proposition: Non-diff example; Intermediate Step}. To record the dependence of the family of intervals $\{I_j\}_{j=1}^N$ on $n$ we denote them by $\{I_j^{(n)} \}_{j=1}^{N_n}$ till the end of proof.

     Clearly $\phi_n$'s are nonnegative and $C^1$. Since the bounds on the right hand side of \eqref{Sup norm Bound of phi_n}  and \eqref{Bound on Holder Norm of phi_n} are summable $\DS \phi:=\sum_{n=1}^{\infty}\phi_n$ is a well defined bounded $C^\theta$ function. Thus condition $(3)$ is verified.

     To verify condition $(1)$, note that by Lemma \ref{Lemma: Linear Response for C1 observables}, for any $n$ and any $\lambda\in(\lambda_0-\epsilon_0,\lambda_0+\epsilon_0)$,

\begin{equation}\label{Computation of derivative with C^1 observable}
    h_n'(\lambda) = \int_{\Sigma} \frac{d}{d\lambda} \phi_n(X(\lambda;a))\,d\nu(a) = \int_{\Sigma} \phi_n'(X(\lambda;a)) X^{(1)}(\lambda;a)\,d\nu(a).
\end{equation}

 Note that the support of $\phi_n'$ is contained in the $\delta'$-neighborhood of $\DS  \bigcup_{j=1}^N I_j$, hence the $2\delta'$-neighborhood of $A_0$. By \eqref{equation:lower bound for A_0}, if 
 $\phi_n'(X(\lambda;a))\!\!\neq\!\! 0$ then 
 $X(\lambda;a)\!\!>\!\! \frac{1}{1-\lambda_0}- \delta' - 2\delta'\!\!>\!\! \frac{1}{1-\lambda}-4\delta'.$ Thus $X^{(1)}(\lambda;a)\ge\frac{1}{2}>0$ by Lemma \ref{Lemma: X^(1) is not too small if X is large}. It follows that $h_n'(\lambda)$ is non-negative, because $\phi_n'$ is also non-negative.

Finally, we  verify condition $(2)$ and wish to bound 
$$   h_n'(\lambda_0) = \int_{\Sigma} \phi_n'(X(\lambda_0;a)) X^{(1)}(\lambda_0;a)\,d\nu(a).$$
By the argument above, the integrand is nonzero only in the region where 
$X^{(1)}(\lambda_0;a)\!\!>\!\!\frac{1}{2}$. Therefore,

\begin{align*}
           h_n'(\lambda_0) &= \int_{\Sigma} \phi_n'(X(\lambda_0;a)) X^{(1)}(\lambda;a)\,d\nu(a)
           \ge \int_{\{a\in\Sigma: X(\lambda_0;a)\in \text{supp}\phi_n'\}} \phi_n'(X(\lambda_0;a)) \frac{1}{2}\,d\nu(a)\\
           &\ge \frac{1}{2}\int_{\{a\in\Sigma: X(\lambda_0;a)\in \cup_{j=1}^N I_j\}} \phi_n'(X(\lambda_0;a)) \,d\nu(a) \\
           &=\frac{1}{2}\int_{\{a\in\Sigma: X(\lambda_0;a)\in \cup_{j=1}^N I_j\}} 1\,d\nu(a) = \frac{1}{2}\mu(\cup_{j=1}^N I_j) \ge \frac{3}{8}C.
     \end{align*}

Since $C\!\!>\!\!0$ is an absolute constant depending only on $\lambda_0$, $\DS \sum_n h_n'(\lambda_0)\!\!=\!\!\infty$ as desired. 
   \end{proof}

   \subsection{Proof of Lemma \ref{Lemma: One blowup implies genericity}}\label{Subsection: proof of lemma: one blow up implies genericity}

       For simplicity, we suppress the subscript and write $\lambda$ for $\lambda_0$. Without loss of generality, assume $\{\epsilon_n\}$ is such that
       \begin{equation}
           \left|\frac{h_\phi(\lambda+\epsilon_n) - h_\phi(\lambda)}{\epsilon_n}\right|>n^2.
       \end{equation}

       For each $n\in\mathbb{N}$, define
       $$B_n:= \left\{\psi\in C^\theta:  \left|\frac{h_\psi(\lambda+\epsilon_n) - h_\psi(\lambda)}{\epsilon_n}|>n\right|
       \right\}
       \quad\text{and}\quad
       G:= \bigcap_{N=1}^\infty \bigcup_{n\ge N} B_N.$$
Clearly $B_n$'s are open subsets of $C^\theta$, and $G$ is a nonempty $G_\delta$ set because $\phi\in G$. Moreover, any $\psi\in G$ satisfies that $h_\psi$ is not differentiable at $\lambda$. 

Therefore it suffices to show that $G$ is dense in $C^\theta$. To prove this fix any $f\in C^\theta$ and $\kappa>0$. We claim that there exists $0\le \kappa'\le\kappa$ such that $f+\kappa' \phi \in G$, which immediately implies density. Note that if $f\in G$ already then the claim is trivial. Thus we assume $f\notin G$.

By definition, $f\notin G$ means that there exists $N_0\in \mathbb{N}$ such that for all $n\ge N_0$, 

$$\left|\frac{h_f(\lambda+\epsilon_n) - h_f(\lambda)}{\epsilon_n}\right|\le n.$$

It follows that for any $\kappa'>0$ and $n\ge N_0$,

\begin{align*}
    \left|\frac{h_{f+\kappa' \phi}(\lambda+\epsilon_n) - h_{f+\kappa'\phi}(\lambda)}{\epsilon_n}\right| &\ge  
    -\left|\frac{h_{f}(\lambda+\epsilon_n) - h_{f}(\lambda)}{\epsilon_n}\right|+ 
    \kappa'\left|\frac{h_{\phi}(\lambda+\epsilon_n) - h_{\phi}(\lambda)}{\epsilon_n}\right|\\
    &\ge -n+ \kappa' n^2
\end{align*}
which will be greater than $n$ if $n\ge N_0$ is sufficiently large. This implies that $f+\kappa' \phi \in G$ for any $\kappa'>0$, and in particular the claim is true.

\section{Proof of a.e. Differentiability}\label{sec: proof of first assertion}

We now turn to the proof of Theorem \ref{thm:phase transition}. Recall (and define) that

\begin{align}
 W_{a,b}(x)&=\sum_{n\ge0}a^n\cos(2\pi b^nx),\notag\\
 P_\lambda(R)&:=\hat{\mu}_\lambda(2\pi R)
       =\prod_{k\ge0}\cos(2\pi R\lambda^k),\notag\\
 h(\lambda)&=\int W_{a,b}\,d\mu_\lambda
       =\sum_{n\ge0}a^nP_\lambda(b^n),
       \qquad 0<\lambda<1.                    \label{eq:response}
\end{align}
 To see the last equality, note that $\cos(x) = \frac{e^{ix} + e^{-ix}}{2}$ and that $\mu_\lambda$ is even (meaning that $\int f(x)\,d\mu_\lambda(x)= \int  f(-x) \,d\mu_\lambda(x)$ for any bounded continuous function). Therefore
 \begin{align*}
     h(\lambda) = \sum_n a^n \int\cos(2\pi b^n x)\,d\mu_{\lambda}(x) &=  \sum_n a^n \int \frac{e^{2\pi i b^nx}+e^{-2\pi i b^n x}}{2}\,d\mu_{\lambda}(x)  \\
     &=  \sum_n a^n \int e^{2\pi i b^nx}\,d\mu_{\lambda}(x)\\
     &=\sum_{n}a^nP_\lambda(b^n).
 \end{align*}
 Moreover, both series defining $W_{a,b}$ and $h$ converge uniformly, so $h$ is continuous. Also recall
\begin{equation}
 \alpha=\frac{-\log a}{\log b}\in(0,1),\qquad
 d(\lambda)=\frac{\log2}{-\log\lambda},\qquad
 \lambda_c=2^{-1/(1-\alpha)}.
\end{equation}

In this section we prove a.e. differentiability in the interval $(\lambda_c,1)$.  Moreover, we will show that the a.e. derivative is equal to the formal derivative, which is an absolutely convergence series:

\begin{equation}\label{eq:derivative}
 h'(\lambda)=\sum_{j=1}^{\infty}a^jD_\lambda(b^j),
 \qquad \sum_{j=1}^{\infty}a^j|D_\lambda(b^j)|<\infty,
\end{equation}
where $D_\lambda$ is defined as
\begin{equation}\label{eq:fourier}
D_\lambda(z)=\partial_\lambda P_\lambda(z)\qquad .
\end{equation}

The core estimate is that, for any $\epsilon>0$, for a.e. $\lambda$, $D_\lambda(b^n)$ is of order at most $(b^n)^{1-d(\lambda) +\epsilon}$ for all sufficiently large $n$. Once this is done the formal derivative $\sum_j a^j D_\lambda(b^j)$ is absolutely convergent, and it is not hard to show it is equal to the actual derivative.

To show the typical scale $(b^n)^{1-d(\lambda)+\epsilon}$, our strategy is to turn the infinite product  in $P_\lambda$ into a sum by considering $-\log|P_\lambda(\xi)|:= \sum_{k\ge 0} g(\lambda^k \xi)$, with $g(t) = -\log|\cos(2\pi t)|$. Since $t$ is taken along a geometric sequence $b^n$, we will first approximate $g$ (in $L^1$) by a real trigonometric polynomial $q$, then use methods from analysis of lacunary sequences to bound the (sixth) moment of $\sum q(\lambda^k \xi)$. The idea is inspired by \cite[Section~3]{aistleitner2023lacunary}. Once this is done, one can obtain its typical size by Markov inequality and some elementary calculation. 

Before going to the proof, we again address how to pass from Theorem \ref{thm:phase transition} to its Sobolev version Theorem\ref{Sobolev counterpart of thm3}. This is again because one can multiply the Weierstrass function $W_{a,b}$ by a smooth cutoff function that $1$ on the common support of $\mu_\lambda$. And it is not hard to show that the resulting function is in $H^s$ for any $s<\alpha:=\frac{-\log a}{\log b}$. 

Here is the plan for the rest of this subsection. The estimate of the (sixth) moment of $\sum q(\lambda^k)$ is carried out in \S \ref{sec: moment bound}. And in \S \ref{sec: Uniform small tail}, \S\ref{sec: Difference at different scale} we upgrade this moment bound into technical estimates that work for a.e. $\lambda$. In \S \ref{sec:concluding the proof} we conclude the proof with these estimates.

\subsection{Moment bound of Real Trigonometric Polynomials}\label{sec: moment bound}

Throughout this subsection, fix $J = [u,v]\subset(0,1)$ to be a compact sub-interval. 

We recall Van der Corput lemma \cite{stein1993harmonic}. 

\begin{lm}[Van der Corput]\label{lem: van der corput}
Suppose $\phi$ is a real-valued smooth function defined on $(u,v)$, such that $|\phi^{(k)}(x)|\ge 1$ on all of $(u,v)$. If either $k\ge 2$ or $k=1$ and moreover $\phi'$ is monotone on $(a,b)$, then there exists constants $c_k>0$ independent of $\phi$ so that for any $A>0$,
$$|\int_{[a,b]} e^{iA\phi(x)}\,dx|\le c_k A^{-\frac{1}{k}}.$$
\end{lm}

\begin{lm}[Sixth Moment of Real Trignometric Polynomial]
\label{lem: 6th moment of trig poly}
    Let $D \in\mathbb{Z}_{>0}$ and let $\DS q(t) = \sum_{0<d\le D} q_d e^{2\pi i dt}$ be a real trigonometric polynomial of degree at most $D$ and with zero mean. Let $R>0$ be a large number, $\lambda\in J$ and $N:=N(\lambda,R):=[\frac{\log R}{-\log\lambda}]+1$. Let $I$ be any interval of integers contained in $\{1,2,\cdots,N\}$ (so $I$ is of the form $ \{a,a+1,\cdots,a+k\}$). Then 
    \begin{equation}\label{eq: 6th moment of trig poly}
        \int_J|\sum_{k\in I} q(\lambda^k R)|^6\,d\lambda = \mathcal{O}_{q,J}(|I|^3)
    \end{equation}
\end{lm}

\begin{proof}
  Since $J$ is a compact subinterval of $[0,1]$, one can fix an $m\in\mathbb{Z}_+$ so that 
  
    \begin{equation}\label{eq:Sparsity condition}
        \lambda^{-m} >10D\,\,;\;\sum_{l\ge m} (l+1)^2\lambda^d<\frac{1}{100D}\,,\quad\forall \lambda\in J.
    \end{equation}
Split $I$ into residue classes modulo $m$, and consider any such residue class $I'$. We have, since $q$ is 
real

\begin{equation}
  |\sum_{k\in I'} q(R\lambda^k)|^6 = \sum q_{d_1}q_{d_2}q_{d_3}q_{d_4}q_{d_5}q_{d_6} e^{2\pi i R \sum_{j=1}^6 d_j \lambda^{k_j}}. 
\end{equation}
Here the summation is over all $d_1,\cdots,d_6\in [-D,D]\cap\mathbb{Z}\setminus\{0\}$, and over all $k_1,\cdots,k_6 \in I'$. Let us write, after combining the terms with the same power in $\sum_{j=1}^6 d_j \lambda^{k_j}$, that this expression is equal to 
$$F(\lambda):=\sum_{j=1}^r c_j \lambda^{l_j} , \quad r\le 6,\,|c_j|\le 5D , \; l_j\in I' \text{ are pairwise distinct}$$
and distinct exponents are such that $|l_j - l_{j'}| \ge m$.

Consider the contribution of $S:=\int_{J}q_{d_1}q_{d_2}q_{d_3}q_{d_4}q_{d_5}q_{d_6} e^{2\pi i R \sum_{j=1}^r c_j \lambda^{l_j}}\,d\lambda$ to the sixth moment. There are two cases:

If $\displaystyle F(\lambda) = \sum_{j=1}^r c_j \lambda^{l_j} = 0$, then $S = \mathcal{O}_{q}(|J|) = \mathcal{O}_q(1)$. Moreover,  by Equation \eqref{eq:Sparsity condition}, the sum can be zero only if all the $c_j$'s are zero, which can only happen if initially, the sum $\sum_{j=1}^6 d_j \lambda^{k_j}$ contains at most three different powers of $\lambda$, and for each distinct power, its coefficient add up to give $0$. The number of such summation is $\mathcal{O}_{q}(|I|^3)$. Therefore the total contribution of such terms is $\mathcal{O}_{q,J}(|I|^3)$.

If $F(\lambda) \neq 0$, then its first and second order derivative are
$$F'(\lambda) = \sum_{j=1}^r c_j l_j \lambda^{l_j-1} , \;\; F''(\lambda) = \sum_{j=1}^r c_j l_j(l_j-1) \lambda^{l_j-2}.$$
By condition \eqref{eq:Sparsity condition}, it is not hard to see that the term with smallest power dominates in both $F',F''$, and in particular determines its sign (unless $F''=0$, but in this degenerate case the subsequent argument holds with little modification) . Therefore $F'$ is monotone and $|F'|\ge c \lambda^{N-l_0}$ for some $c$ depending only on $J$ and $q$, and $N-l_0$ denotes the smallest power. By Lemma \ref{lem: van der corput}, 
$$|S|=\mathcal{O}_{q,J}((R\lambda^{N-{l_0}})^{-1}) = \mathcal{O}_{q,J}(\lambda^{l_0}). $$
It remains to upper bound the number of such situation. Since $\lambda^{N-l_0}$ is the surviving term in $\sum_{j=1}^6 d_j \lambda^{k_j}$ with least power and we have \eqref{eq:Sparsity condition}, the sum must take one of the following forms:
\begin{itemize}
    \item All powers appeared in the sum are at $\ge N-l_0$, and at least one of them is equal to $N-l_0$. The number of such summation is $\mathcal{O}_{q}(l_0^5)$.
    \item There are powers $<N-l_0$ in the sum but their coefficients cancel, and there is at least one term with power $N-l_0$. The number of such situation is $\mathcal{O}_q(|I'|^2 + |I'|^2 l_0 + |I'|l_0^2 + |I'|l_0^3)$. Let us briefly explain why the bound holds. First assume there are exactly one term in $\sum_{j=1}^6 d_j \lambda^{k_j}$ that has power $N-l_0$, and there are $a$ terms with power $<N-l_0$ and $5-a$ terms with power above $N-l_0$. Since the coefficients for $<N-l_0$-power terms need to cancel, $a$ can only be $5,4,3,2$. Moreover, there are only $l_0$ choices for each term with power $>N-l_0$, thus one gets the aforementioned bound. The case of there are more than one terms with power exactly $N-l_0$ is similar, and the bounds are absorbed in the terms we already get.
\end{itemize}

Summing up, the total contribution in the case $F(\lambda)\neq0$ is 
$$\mathcal{O}_{q,J} (\sum_{l_0\in [0,N] , N-l_0 \in I'} \lambda^{l_0} (l_0^4 + |I'|^2 + |I'|^2l_0+|I'|l_0^2 + |I'|l_0^3) ) = \mathcal{O}_{q,J}(|I'|^3).$$

The foregoing discussion gives
$\DS \int_J  |\sum_{k\in I'} q(R\lambda^k)|^5\,d\lambda = O_{q,J}(|I'|^2).$
By the triangle inequality in $L^5$ and the fact that $m$ is fixed, this finishes the proof.
\end{proof}

With a bound on the sixth moment, one obtains an asymptotic upper bound of the tail sum for a.e. $\lambda\in J$, if $R$ is taken along an exponentially growing sequence. More precisely:

\begin{lm}\label{lem: log smallness of tail sum}
    Let $R_n = b^n , n=1,2,\cdots$ for any fixed $b>1$. Let $q,\lambda$ be as in Lemma \ref{lem: 6th moment of trig poly}, and $N_n = N_n(\lambda,b) = \lfloor\frac{\log R_n}{-\log\lambda}\rfloor+1$. Then for a.e. $\lambda\in J$, 
    $$\max_{1\le K\le N} |\sum_{k=K}^N q(\lambda^k R_n)| = o_{q,b,J}(\log R_n) \text{ as }n\to\infty.$$
\end{lm}

\begin{proof}
    By Lemma \ref{lem: 6th moment of trig poly}, 
    $$\int_J\max_{1\le K\le N} |\sum_{k=K}^N q(\lambda^k R_n)|^6\,d\lambda \le \sum_{K=1}^N \int_{J} |\sum_{k=K}^N q(\lambda^k R_n)|^6\,d\lambda = \mathcal{O}_{q,J}(N_n^4).$$
    By Markov inequality, 
    $$\text{Leb} \{\lambda\in J: \max_{1\le K\le N} |\sum_{k=K}^N q(\lambda^k R_n)| \ge N_n^{\frac{11}{12}}\}   \le \mathcal{O}_{q,J}(N_n^{4-6\frac{11}{12}})\le \mathcal{O}_{q,J}(N_n^{-1.5}).$$
    The right hand side is summable in $n$, so Borel-Cantelli implies for a.e. $\lambda\in J$, there exists $n_0(\lambda)\in\mathbb{N}$ so that for all $n\ge j_0(\lambda)$, 
    $\DS \max_{1\le K\le N} |\sum_{k=K}^N q(\lambda^k R_n)| < N_n^{\frac{11}{12}} = o_{q,J,b}(\log R_n).$
\end{proof}

\subsection{Uniform Small Tails}\label{sec: Uniform small tail}

Let $g(t) = -\log|\cos(2\pi t)|$. First note that $g\in L^1([0,1])$ and $\int_{[0,1]} g(t)\,dt = \log2$. To see this, first substitute $x=2\pi t$ and by symmetry,
\[
\int_0^1 \log|\cos(2\pi t)|\,dt
=\frac{1}{2\pi}\int_0^{2\pi}\log|\cos x|\,dx
=\frac{2}{\pi}\int_0^{\pi/2}\log(\cos x)\,dx.
\]
Set
\[
J:=\int_0^{\pi/2}\log(\cos x)\,dx.
\]
The substitution $x\mapsto \frac{\pi}{2}-x$ gives
\[
J=\int_0^{\pi/2}\log(\sin x)\,dx.
\]
Consequently, using $\sin x\cos x=\frac12\sin(2x)$,
\begin{align*}
2J
&=\int_0^{\pi/2}\log(\sin x\cos x)\,dx\\
&=-\frac{\pi}{2}\log 2
  +\frac12\int_0^\pi\log(\sin u)\,du\\
&=-\frac{\pi}{2}\log 2+J,
\end{align*}
where the last equality follows from symmetry about $\pi/2$. Thus $J=-\frac{\pi}{2}\log 2$, proving the claim.

Now that $g$ is integrable, one can approximate it by real trignometric polynomials in $L^1$ from below. More precisely,  fix any $\delta>0$, there exists a real trigonometric polynomial $q$ such that $q\le g$ and $\bar{q}:=\int_{[0,1]}q \ge \int_{[0,1]} g(t)\,dt - \delta =  \log2-\delta$. Let $R_n = b^n , N_n = N_n(\lambda,b) = \frac{\log R_n}{-\log\lambda}$ as above. Then by Lemma \ref{lem: log smallness of tail sum}, for any $\epsilon>0$, for  a.e. $\lambda\in J$, there exists $n_0(\lambda)\in\mathbb{N}$ so that for all $n\ge n_0$,

$$|\sum_{k=K}^{N_n}(q(R_n\lambda^k) - \bar{q})| \le\epsilon \log R_n,\;\forall 1\le K\le N_n.$$
Therefore, 
\begin{align}\label{eq: small tail}
 \nonumber  -\log( \prod_{k= K}^{N_n} |\cos(2\pi R_n \lambda^k)|) &\ge \sum_{k=K}^{N_n} q(R_n\lambda^{k})\\
 \ge -\epsilon\log R_n + (\log2-\delta)(N_n-K+1)
 &\ge (N_n-K+1)\log2 -\left(\frac{\delta}{-\log\lambda}+\epsilon\right) \log R_n.
\end{align}

The above inequality gives a uniform upper bound on
$\DS \left|\prod_{k=K}^{N_n} \cos(2\pi R_n\lambda^k)\right|$ for $1\le K\le N_n$ and all sufficiently large $n$. This is what we mean by uniform small tail as in the title of this subsection:

\begin{lm}
\label{lem: small tail}
    For any $\delta>0$ and $J\Subset(\lambda_c,1)$, for a.e. $\lambda\in J$, there exists $n_0(\lambda)$ such that for all $n>n_0(\lambda)$ , 
    \begin{equation}
        |\prod_{k=K}^{N_n} \cos(2\pi R_n \lambda^k)| \le 2^{-(N_n-K+1)} R_n^{{ +\delta}} , \quad\forall1\le K\le N_n.
    \end{equation}
\end{lm}

 To estimate derivative, we also need to bound each term in the product from below.

\begin{lm}\label{lm: small single term}
    For any $\tau>0$, for a.e. $\lambda\in J$ there exists $n_1(\lambda)\in\mathbb{N}$ so that
\begin{equation}\label{eq: small single term}
    |\cos(2\pi R_n\lambda^k)|\geq R_n^{-\tau},\forall n\ge n_1(\lambda) , \; \forall k\ge1,
\end{equation}

\end{lm}

\begin{proof}
    First note that there exists $p_0\in\mathbb{N}$, depending only on $J$ and $b$, so that one only needs to show \eqref{eq: small single term} for $0\le k\le N_n+p_0$, as $|2\pi R_n\lambda^k|\le \frac{\pi}{4}$ for all $k> N_n+p_0$, and the corresponding cosine value will be of order $1$.

Now for each $1\le k\le N_n+p_0$, consider $E_k:=E_k(J,b,\tau):=\{\lambda\in J:|\cos(2\pi b^n \lambda^k)|\le R_n^{-\tau}\}$. We aim to bound the Lebesgue measure of $E_k$.

Note that the interval traversed by $b^n\lambda^k , k\le N_n+p_0  , \lambda\in J$ is bounded from below, and $\cos$ is a smooth periodic function. Therefore,  if one defines
$$E_k':=E_k'(J,b,\tau): = \{t : t= b^n\lambda^k \text{ for some }\lambda\in J \text{ and } |\cos(2\pi b^n \lambda^k)|\le R_n^{-\tau}\},$$
then $\text{Leb}(E_k')$ is bounded by $R_n^{-\tau}$ times the total length of $\{b^n\lambda^k: \lambda\in J\}$ (modulo a constant multiple). Recall that $J = [u,v]$, the above argument gives
$$ \text{Leb}(E_k')\le C_0b^{-\tau n}b^n  kv^{k-1}$$
for some constant $C_0>0$ depending only on $J,b$. Note that $E_k = f_1\circ f_2(E_k')$, with $f_2(x) = b^{-n}x$ and $f_1(x) = x^{\frac{1}{k}}$. Moreover, $|f_1'| \le \frac{C_0' v^{-k+1}}{k}$ for some $C_0'>0$ depending only on $J$, when restricted to $f_2(E'_k)\subset \{\lambda^k:\lambda\in J\}$. Therefore, 
$$\text{Leb}(E_k)\le \frac{C_0'v^{-k+1}}{k} b^{-n} C_0 b^{-\tau n}b^n  kv^{k-1} = C_0'' b^{-\tau n} .$$
Summing over $k$ and noting that $N_n = O_{b,J}(n)$ implies 
$$\text{Leb} (\{\lambda\in J:   |\cos(2\pi R_n\lambda^k)|\geq R_n^{-\tau} \text{ for some } k\ge 1\}) \le C_0 b^{-\tau n}n, \; C_0 = C_0(J,b)$$
and the right hand side is summable in $n$. By Borel-Cantelli, the lemma is proved.
\end{proof}

\subsection{ Oscillations at Small Scales}\label{sec: Difference at different scale}

 Recall \eqref{eq:fourier}.
In this subsection we use the above estimates to prove typical order of $P_\lambda(b^n)$ and $D_\lambda(b^n)$.

\begin{prop}\label{prop: typical order}
Fix any $\epsilon>0,p\in(0,1)$and compact subinterval $J\Subset (0,1)$. For a.e. $\lambda\in J$, there exists $n_0(\lambda)\in\mathbb{N}$ so that for all $n>n_0(\lambda)$, the following holds:
    \begin{align}\label{eq:typical scale}
 \sup_{|\nu-\lambda|\leq r}|P_\nu(R_n)|
   &\leq C_0 R^\epsilon r^{d(\lambda)}
       &&(\forall\, R_n^{-1}\leq r\leq R_n^{-p}),\\
 \sup_{|\nu-\lambda|\leq R_n^{-1}}|D_\nu(R_n)|
   &\leq C_0 R_n^{1-d(\lambda)+\epsilon}.
\end{align}

\end{prop}

\begin{proof}
    Write $d=d(\lambda)$ and $\DS d_+=\max_{\lambda\in J}d(\lambda)$. Choose $\tau>0$ so small that
\begin{equation}\label{eq:tau-choice}
 3\tau<p,\qquad \tau(1+3d_+)<\epsilon.
\end{equation}
Work with $\lambda\in J$ so that there exists $n_0(\lambda)$ so that for all $n>n_0(\lambda)$, 
inequalities \eqref{eq: small tail} and \eqref{eq: small single term}  hold. Such $\lambda$ has full measure, and we will show \eqref{eq:typical scale} for this $\lambda$ and all $n>n'_0(\lambda)$, where $n_0'(\lambda)>n_0(\lambda)$ will be specified along the way. We thus fix this $\lambda$ and start with an $n>n_0(\lambda)$, and abbreviate $R_n := b^n$ as $R$. Throughout this proof, $C_0$ will denote implicit constants depending only on $p,\epsilon,c,C,J$ but not on $n$.

Lemma \ref{eq: small tail} gives small tail product in $P_\lambda(z)$ at our fixed $\lambda\in J$. To show smallness of $P_\nu(z)$ for nearby $\nu$, we will consider the quotient
$$\frac{\prod_{k\ge K}|\cos(2\pi R\nu^k)|}{\prod_{k\ge K}|\cos(2\pi R\lambda^k)|}$$
and show it is bounded for a suitable $K$. 

It is important to truncate at an appropriately large $K$ and compare the tail product. This is because we wish to apply the mean value theorem, $\DS |\cos(2\pi R\nu^k) - \cos(2\pi R\lambda^k)|\le  2\pi R k\tilde{\lambda}^{k-1}\le 2\pi R k (\lambda+r)^{k-1}$(here $\tilde{\lambda}$ is some number between $\lambda$ and $\nu$), and the right hand side will be small only if $k$ is large.

More precisely,  for $R^{-1}\leq r\leq R^{-p}$ define
\begin{equation}\label{eq:K-local}
 K=\left\lceil\frac{\log(Rr)+3\tau\log R}{-\log\lambda}\right\rceil.
\end{equation}
For large $R$, $1\leq K\leq N$, uniformly for $r\in[R^{-1} , R^{-p}], \lambda\in J$, and
\begin{equation}\label{eq:K-properties}
 \lambda^K\leq(Rr)^{-1}R^{-3\tau},\qquad K\leq C_0\log R,
 \qquad (\lambda+r)^K\leq C_0\lambda^K.
\end{equation}
The last bound follows from $(1+r/\lambda)^K\leq\exp(Kr/\lambda)$, because $Kr\leq C_0R^{-p}\log R\to0$.

Put $A_k=|\cos(2\pi R\lambda^k)|\geq R^{-\tau}$ by \eqref{eq: small single term}. The mean value theorem gives, for $|\nu-\lambda|\leq r$,
\begin{align}
 \sum_{k\geq K}
 \frac{|\cos(2\pi z\nu^k)-\cos(2\pi z\lambda^k)|}{A_k}
 &\leq C_0 R^{1+\tau}r\sum_{k\geq K}k(\lambda+r)^{k-1}\notag\\
 &\leq C_0R^{1+\tau}rK\lambda^K
 \leq C_0KR^{-2\tau}\label{eq:relative-tail}
\end{align}
which is $o(1)$ as $R\to\infty$. This  implies
\begin{align}
    \nonumber \frac{\prod_{k\geq K}|\cos(2\pi z\nu^k)|}{\prod_{k\geq K}A_k}&\le
    \prod_{k\ge K} \left[1+\frac{|\cos(2\pi z\nu^k) - \cos(2\pi z \lambda^k)|}{A_k}\right]\leq 2
\end{align}
for all sufficiently large $R$ (this is where we need to enlarge $n$), which in turn gives
$$
    |P_\nu(z)|\le \prod_{k\ge K}|\cos(2\pi z\nu^k)| \le 2\prod_{k\ge K}A_k 
    \le  C_0 R^\tau 2^{-(N-K+1)}
    \le C_0 R^{\tau - d(\lambda)} (Rr)^{d(\lambda)} R^{3\tau d(\lambda)}
    \le C_0 R^\epsilon r^{d(\lambda)}.
$$

The third inequality above follows from Lemma \ref{lem: small tail}. The last inequality holds by \eqref{eq:tau-choice}.

For the derivative bound take $r=R^{-1}$ and split $P_\lambda(z)=H_\lambda (z)T_\lambda(z)$ into the head $k<K$ terms and tail $k\geq K$ terms, where $K$ is the same as in \eqref{eq:K-local}. Trivially one has  $|\partial_\lambda H_\lambda(R)|\leq C_0R$, $|H_\lambda(R)|\leq1$. Moreover, $|T_\lambda(R)|\le C_0R^\epsilon r^{d(\lambda)}$ by the previous argument.

Finally, we have, for any $|\nu-\lambda|\le R^{-1}$,
\begin{align*}
  &   |\partial_\lambda T_\lambda|_{\lambda =\nu}(R)|
 \leq C_0|T_\lambda|_{\lambda = \nu}| |R|\sum_{k\ge K} k\nu^{k-1}| \frac{\sin(2\pi R\nu^{k})}{\cos(2\pi R\nu^k)} |\\
 &\le C_0 R^{1+\epsilon - d(\lambda)+\tau} K(\lambda+r)^K
 \le C_0 R^{1+\epsilon - d(\lambda)+\tau} R^{-3\tau}
 \leq C_0R^{1+\epsilon - d(\lambda)}.
\end{align*}
Here in the second line, we used the fact that $|\cos(2\pi z\lambda^k)|\ge R^{-\tau}$ and $\frac{|\cos(2\pi z\nu^k)|}{|\cos(2\pi z\lambda^k)|}\le 2$ if $|\nu-\lambda|\le R^{-1}$ and $R$ is sufficiently large. 

Substituting the above estimates into
$\displaystyle \partial_\lambda P_\lambda=(\partial_\lambda H_\lambda)T_\lambda+H_\lambda\partial_\lambda T_\lambda $
proves the proposition.
\end{proof}

\subsection{Concluding the Proof}\label{sec:concluding the proof}

We are now ready to prove the a.e. differentiability assertion.

\begin{proof}
    
We show that for any $J\Subset (\lambda_c ,1)$, for a.e. $\lambda\in J$ the response function is differentiable at $\lambda$. The conclusion will immediately follow.

Fix $\lambda\in J$ so that for all $n>n_0(\lambda)$, the conclusion of Proposition \ref{prop: typical order} holds with the choice
\begin{equation}\label{eq:upper-epsilon}
 p = \frac{\alpha}{2};\quad0<\epsilon<\min_{\lambda\in J}\{\alpha,\alpha+d(\lambda)-1\}.
\end{equation}
Such $\lambda$ has full measure. With this $\lambda$ fixed we abbreviate $d = d(\lambda)$.

Let  $h_j(\nu)=a^j P_\nu(R_j) =  R_j^{-\alpha}P_\nu(R_j)$. Then 
$\DS h(\lambda) = \sum_{j=1}^\infty h_j(\lambda)$. For small nonzero $t$, 
\begin{equation}\label{eq:three-ranges}
 \left|\frac{h_j(\lambda+t)-h_j(\lambda)}t\right|
 \leq C_0
 \begin{cases}
 R_j^{1-\alpha-d+\epsilon},
            & |t|\leq R_j^{-1},\\[2pt]
 R_j^{-\alpha+\epsilon}|t|^{d-1},
            & R_j^{-1}<|t|\leq R_j^{-p},\\[2pt]
 R_j^{-\alpha}|t|^{-1},
            & |t|>R_j^{-p}.
 \end{cases}
\end{equation}
The first line uses the mean value theorem and then the derivative bound in Proposition \ref{prop: typical order}. The second uses the bound of $|P_\lambda(z)|$ from the same Proposition at both $\lambda+t$ and $\lambda$ with $r=|t|$. The third uses $|P_{\lambda}|\le 1$.

Now split the sum into
$$\frac{h(\lambda+t) - h(\lambda)}{t} = \sum_{j\ge1}\frac{h_j(\lambda+t) - h_j(\lambda)}{t}= \sum_{j:R_j\le |t|^{-1}}  + \sum_{j: |t|^{-1}\le R_j\le |t|^{-\frac{1}{p}}}  + \sum_{j: R_j > |t|^{-\frac{1}{p}}}. $$

By the estimate in the first case in \eqref{eq:three-ranges}, as $t\to0$, the first sum in the above right hand side is absolutely convergent, since $1-\alpha-d+\epsilon<0$, with the limit equal to the formal derivative $\sum_j a^j D_\lambda(b^j\xi_j)$. The second sum is bounded by
$$\sum_{j: |t|^{-1}\le R_j\le |t|^{-\frac{1}{p}}}  R_j^{-\alpha+\epsilon} |t|^{d-1}  \le C_0 |t|^{\alpha+d-1-\epsilon}$$
which tends to $0$ as $t\to0$. The third sum is bounded by 
$$\sum_{j: R_j > |t|^{-\frac{1}{p}}} R_j^{-\alpha}|t|^{-1} \le C_0|t|^{-1+\frac{\alpha}{p}} \le C_0 |t| $$
which also tends
to $0$ as $t\to 0$. This finishes the proof of almost everywhere differentiability.
\end{proof}

\section{Proof  a.e. Non-differentiability}\label{sec: proof of second assertion}

This subsection proves the (a.e.) non-differentiability assertion in Theorem \ref{thm:phase transition}. The starting point is a sufficient condition for non-differentiability in terms of Fourier localization, which will be stated and proved in \S \ref{subsec:suff.cond for non-diff}. This method is well adapted to the special structure of the response function $h$, which we recall \eqref{eq:response}:
$$ h(x)=\int W_{a,b}\dd\mu_x  =\sum_{n\ge0}a^nP_x(b^n).$$
This is a sum of Fourier transform of $\mu_x$ along a geometric growing sequence $\{b^n\}$, and we will exploit this structure by localizing around $b^n$ to single out the oscillation caused by $a^n P_x(b^n)$, which will blow up and lead to non-differentiability. 

Below are some notations. We will apply Fourier transforms in the \emph{$x$ variable}, sand to emphasis this we will use $x$ to denote $x$, the parameter for Bernoulli convolution, till the end of this section. Let $\mathcal{F}(f)(\xi):= \widehat f(\xi):=\int_\R f(x)e^{-2\pi\ii x\xi}\dd x$. Let $\Omega:=\{-1,1\}^{\N}$ endowed with the product probability $\mathbb{P} = (\frac{1}{2} , \frac{1}{2})^\mathbb{N}$. Moreover, for $x\in(0,1)$, define
\begin{equation}
 F_\omega(x)=\sum_{k\ge1}\omega_kx^k,\qquad
 V_\omega(x)=F_\omega'(x),\qquad
 \cV_x=\{V_\omega(x):\omega\in\Omega\},\qquad
 B(x)=\frac1{(1-x)^2}.                                  \label{eq:slopes}
\end{equation}
Thus $\cV_x\subset[-B(x),B(x)]$, its maximum is $B(x)$, and
\begin{equation}
 P_x(R)=\cos(2\pi R)\E e^{2\pi\ii RF_\omega(x)}.
                                                        \label{eq:symbolic}
\end{equation}
Finally, for $K\in\mathbb{Z}_{\ge 0}$, write
\begin{equation}
 C_K=\{\omega\in\Omega:\omega_1=\cdots=\omega_K=1\},\quad
 S_K(x)=\sum_{k=1}^Kx^k,\quad
 Q_{K,R}(x)=\prod_{k>K}\cos(2\pi Rx^k).                  \label{eq:cylinder}
 \end{equation}
with the usual empty-sum convention. 

Here is the plan of the rest of this section. Several preliminary lemmas will be given in the sequel: In \S \ref{subsec:suff.cond for non-diff} we prove the sufficient condition for non-differentiability in terms of Fourier cutoff; In \S\ref{subsec: separation condition} we impose a technical separation condition that will be satisfied for a.e. $x$; In \S \ref{subsec: indep cos factor} we address the issue of the $x$-independent factor $\cos(2\pi b^n)$ in $\hat{\mu}_x(b^n)$; In \S \ref{subsec: typical order} we derive the order of $\hat{\mu}_x(b^n)$ for a.e. $x$ and $n$ along a subsequence. In \S\ref{subsec: ibp} we present an easy but important integration by parts lemma. Finally, in \S\ref{sec: proof, non-diff} we conclude the proof.

We end by pointing out that, since we are dealing with the case $\alpha+\frac{\log2}{-\log x}<1$, one automatically gets $x<\frac{1}{2}$. Thus below we always assume $x\in (0,\frac{1}{2})$.

\subsection{Fourier localization and pointwise differentiability}\label{subsec:suff.cond for non-diff}

The purpose of this subsection is to derive a sufficient condition for 
 a function $f$ to be non-differentiable at a given point $x$, in terms of the behavior of its Fourier cutoff. 

Fix any Schwartz function $\psi\in C_c^\infty(\R\setminus\{0\})$ (meaning it is compactly supported away from zero). Let $\kappa=\mathcal F^{-1}\psi$ and define, for any bounded function $f$,
\begin{equation}
 \Pi_Rf(x)=\int_\R\kappa(u)f(x-u/R)\dd u,\qquad R>0.
                                                        \label{eq:multiplier}
\end{equation}
When $f$ is also integrable, the above is equivalent to the Fourier multiplier with symbol $\psi(\xi/R)$. That is, 
$$\mathcal{F}(\Pi_R f)(\xi) = \psi\left(\frac{\xi}{R}\right) \hat{f}(\xi).$$

Since $\psi$ vanishes near zero, by Fourier inversion, 
\begin{equation}
  \label{DerFT}  
 \int\kappa(u)\dd u=0,\qquad \int u\kappa(u)\dd u=0.
\end{equation}

\begin{lm}\label{lem:pointwise}
Let $f$ be  a bounded function on $\mathbb{R}$. If
\[
 \limsup_{t\to0}\frac{|f(x+t)-f(x)|}{|t|}<\infty,
\]
then $R|\Pi_Rf(x)|$ is bounded as $R\to\infty$. If $f$ is differentiable at $x$, then $R\Pi_Rf(x)\to0$. Consequently,
\[
 \limsup_{R\to\infty}R|\Pi_Rf(x)|=\infty
 \quad\Longrightarrow \quad\limsup_{t\to0}\frac{|f(x+t)-f(x)|}{|t|}<\infty.
\]
\end{lm}

\begin{proof}
In the first case, choose $C,\delta>0$ such that
$|f(x+t)-f(x)|\le C|t|$ for $|t|<\delta$.  By \eqref{DerFT}

\begin{align}
    \nonumber R|\Pi_Rf(x)|&=R|\int \kappa(u) \left(f(x )- f\left(x-\frac{u}{R}\right)\right)\,du|   \\
   \nonumber &=|\int_{\{|u|<\delta R\}} C|u\kappa(u)| \,du|   + 2R\|f\|_\infty\int_{|u|\ge\delta R}|\kappa(u)|\,d u \\
  \nonumber  &\le C\int|u\kappa(u)|\dd u
     +2R\|f\|_\infty\int_{|u|\ge\delta R}|\kappa(u)|\dd u=O(1),
\end{align}
because $\kappa$ is Schwartz. This proves the first assertion.

At a differentiability point, for any $\eps>0$, there exists $\delta>0$ so that for any $|u|\le \delta$, $|f(x+u) - f(x) - f'(x)u|\le \eps|u|$. Thus 
\begin{align}
    \nonumber R|\Pi_Rf(x)|&=R|\int \kappa(u) (f(x )- f(x-\frac{u}{R}) - f'(x) \frac{u}{R})\,du|   \\
   \nonumber &=|\int_{\{|u|<\delta R\}} \eps|u\kappa(u)| \,du|   + R\int_{|u|\ge\delta R}|\kappa(u)|(2||f||_\infty + |f'(x)\frac{|u|}R{}|)\,d u \\
  \nonumber  &\le O(\eps) + o_{R\to\infty} (1).
\end{align}
Here the last step holds again since $\kappa$ is Schwartz. Since $\eps>0$ is arbitrary, $R\Pi_R f(x)\to 0$ as desired.
\end{proof}

Below the Fourier transform will be taken in the variable $x$, which is the parameter of the Bernoulli convolution. 

\subsection{Almost-everywhere separation and the Fourier cutoff}\label{subsec: separation condition}

In this subsection we establish some separation properties that will hold for a.e. $x\in (0,x_c)$. Later when we show a.e. non-differentiability, these properties will be required. 

\begin{lm}\label{lem:null}
For every $c\in(-1,1)$, the set
\[
 Z_c=\{x\in(0,\tfrac12):V_\omega(x)=cB(x)
                         \text{ for some }\omega\in\Omega\}
\]
has Lebesgue measure zero.
\end{lm}

\begin{proof}
Fix $I:=[u,v]\Subset(0,\tfrac12)$ and write
\[
 H_\omega(x)=\sum_{k\ge1}k(\omega_k-c)x^{k-1},\qquad
 H_{\omega,N}(x)=\sum_{k=1}^Nk(\omega_k-c)x^{k-1}.
\]
Every $H_\omega$ is analytic and is not the zero function, since $H_\omega(0)=\omega_1-c\ne0$. Thus if $x_0$ is a zero of $H_\omega$ for some $\omega\in\Omega$, then it is contained in
\[
 E_{r,j}=\{x\in I:H_\omega^{(r)}(x)=0,\
       |H_\omega^{(r+1)}(x)|\ge j^{-1}
       \text{ for some }\omega\},
\]
for some integers $r\ge0$ and $j\ge1$. These sets are compact, by compactness of $I\times\Omega$ and uniform convergence of the derivative series. Moreover, for any fixed $r$,
\[
 \|H_\omega^{(r)}-H_{\omega,N}^{(r)}\|_{L^\infty(I)}
 \le C_{I,r,c}N^{r+1}{ v^N}=:d_N(r),
\]
and $d_N(r)\to 0$ as $N\to\infty$. Similar conclusions hold for $|H_\omega^{(r+1)}-H_{\omega,N}^{(r+1)}\|_{L^\infty(I)}$. Consequently, for any fixed $r,j$, for sufficiently large $N$, $E_{r,j}$ is covered by $2^{N+1}-1$ sets of the form
\[
 \{|p|\le d_N,\ |p'|\ge(2j)^{-1}\},\qquad
 p=H_{\omega,N}^{(r)}\text{ for some } \omega\in \Omega.
\]
Note that for any polynomial $p$ of degree at most $N$,
\begin{equation}
 \Leb\{|p|\le d,\ |p'|\ge\eta\}\le\frac{2Nd}{\eta}.
                                                        \label{eq:sublevel}
\end{equation}
To see this, we only need to consider non-constant $p$. On each of at most $N$ monotonicity intervals of $p$, 
changing variables from $x$ to $p(x)$ on the subset on the left hand side gives the estimate. Therefore, 
\[
 \Leb(E_{r,j})\le C_{I,r,c,j}N^{r+2}( 2v)^N\longrightarrow0
\]
and 
$\DS \Leb(\bigcup_{\omega\in\Omega}\{\text{zeros of  }H_w\}) \le \Leb(\bigcup_{r,j\in\mathbb{N}}E_{r,j}) = 0.$
\end{proof}

By Lemma~\ref{lem:null} (with $c=0$ and with $c=b^{-\ell}$, $\ell\ge1$), there is a full-measure set $G_b\subset(0,1/2)$ such that
\begin{equation}
 0\notin\cV_x,\qquad
 B(x)\notin b^\ell\cV_x\quad(\ell\in\Z\setminus\{0\}),
 \qquad x\in G_b.                                      \label{eq:separation}
\end{equation}
Note that when $0\notin \cV_x$, for $ B(x)\notin b^\ell\cV_x,\forall l\in\mathbb{Z}\setminus\{0\}$ to hold one only needs it to be true for finitely many $l$. This is because when $\ell<0$, it is automatic since $\max b^\ell\cV_x=b^\ell B(x)<B(x)$, and for sufficiently large $l$ so that $b^l \min \cV_x >B(x)$ it also holds automatically. In particular, the conditions in \eqref{eq:separation} survive under small perturbation.

\begin{lm}\label{lem:cutoff}
For each $x_0\in G_b$ there are an integer $K\ge1$, an open interval $U\ni x_0$ with $\overline U\subset(0,1/2)$, a function $\psi\in C_c^\infty((0,\infty))$ with $0\le\psi\le1$, and $d>0$ such that, for $x\in\overline U$,
\begin{align}
 |V_\omega(x)-z|&\ge d(1+|z|)
 &&(\omega\in C_K,\ z\in\supp(1-\psi)),                 \label{eq:pass}\\
 |V_\omega(x)-z|&\ge d
 &&(\omega\notin C_K,\ z\in\supp\psi),                 \label{eq:reject}\\
 |b^\ell V_\omega(x)-z|&\ge d(1+b^\ell)
 &&(\ell\ne0,\ \omega\in\Omega,\ z\in\supp\psi).
                                                        \label{eq:cross}
\end{align}
In particular, $\psi=1$ on a neighborhood of all selected cylinder
slopes $V_\omega(x)$, $\omega\in C_K$.
\end{lm}

\begin{proof}
By the argument proceeding to the lemma, if $U$ is sufficiently small, then $\{B(x):x\in \bar{U}\}$ and
$\{b^\ell\cV_{x} : l\in\mathbb{Z}\setminus\{0\} , x\in\bar{U}\}$ are separated. In particular, there exists a neighborhood $A$ of $B(x_0)$ that contains $\{B(x):x\in \bar{U}\}$  and is separated from $\{b^\ell\cV_{x} : l\in\mathbb{Z}\setminus\{0\} , x\in\bar{U}\}$.

Set $w_K(x)=Kx^{K-1}$ and $T_K(x)=\sum_{k>K}kx^{k-1}$. Since $ \bar{U}\Subset (0,\frac{1}{2})$, 
\[
 \frac{T_K(x)}{w_K(x)}
 =\frac{x}{1-x}+\frac{x}{K(1-x)^2}<1
\]
uniformly in $x\in \bar{U}$ for sufficiently large $K$. Moreover,
\begin{align*}
 V_\omega(x)&\in[B(x)-2T_K(x),B(x)]&& (\omega\in C_K),\\
 V_\omega(x)&\le B(x)-2w_K(x)&& (\omega\notin C_K).
\end{align*}
Choose $K$ so that uniformly in $x\in \bar{U}$,
$$T_K(x)<w_K(x),\quad [B(x)-2T_K(x),B(x)]\Subset A$$
and in turn choose $\psi$ so that 
\begin{enumerate}
    \item It is equal to one on a neighborhood of $\bigcup_{x\in\bar{U}}[B(x)-2T_K(x),B(x)]$;
    \item $\supp\psi \Subset A$;
    \item $\supp \psi$ is strictly above $\{B(x)-2w_K(x):x\in\bar{U} \}$.
\end{enumerate}

The lemma now follows with this choice of $\psi$ and a sufficiently small $d>0$. The second assertion is immediate. The first and third hold with the above separation properties, together with the fact that if $|z|>>1$, then $|V_w(x) - z|\approx |z|$.
\end{proof}

\subsection{The parameter-independent cosine factor}\label{subsec: indep cos factor}

A small nuance in our definition of Bernoulli convolution is that the sum $\sum_{k\ge0} \eps_kx^k$ starts at $k=0$, which is independent of $x$. Consequently, 
$\DS \int e^{2\pi\ii Ry}\dd\mu_x(y)
       =\cos(2\pi R)\prod_{k\ge1}\cos(2\pi Rx^k)$ contains a term $\cos(2\pi R)$ that might destroy our lower bound. In this subsection we show it will in fact cause no harm when $R = b^n$ is taken along an exponential sequence. This is due to 

\begin{lm}\label{lem:cosine}
For every $b>1$ and $\eps>0$, there are infinitely many $n\ge0$ such that
\begin{equation}
 |\cos(2\pi b^n)|\ge b^{-n\eps}.                         \label{eq:cosine}
\end{equation}
\end{lm}

\begin{proof}
For every $y\in\R$,
\[
 \dist(4y,\Z)\le\tfrac12|\sin(4\pi y)|
                    \le|\cos(2\pi y)|.
\]
If \eqref{eq:cosine} failed eventually, then $\dist(4b^n,\mathbb{Z})$ would be square summable in $n$. 
By \cite[Theorem~1.5]{schleischitz2014rateaccumulationalphazetanmod}, $b$ is a Pisot number, and $\dist(b^n,\mathbb{Z})$ converges exponentially fast to $0$. This implies $|\cos(2\pi b^n)|\to1$, a contradiction. 
\end{proof}

\subsection{Logarithmic tail products and subsequence lower bounds}\label{subsec: typical order}

This subsection is devoted to show that for any $K\in\mathbb{Z}_+$, for a.e. $x\in (0,\frac{1}{2})$, the product 
$\DS Q_{K,R}(x):=|\prod_{k>K} |\cos(2\pi b^n x^k)|$ will have the ``right'' size of $b^{-n d(x)\pm\epsilon}$  for arbitrarily small $\eps$, with $d(x)$ given by \eqref{DefPar}, as long as $n$ is sufficiently large.
This is stronger than what we did in Lemma \ref{lem: small tail}, because we need an additional lower bound. More precisely, we want to show:

\begin{lm}\label{lem:logarithms}
For every $I=[u,v]\Subset(0,1)$ and fixed $K\ge0$,
\begin{equation}
 \frac{-\log|Q_{K,R}(x)|}{\log R}\to d(x)
       \quad\text{in }L^1(I)\quad(R\to\infty).            \label{eq:logL1}
\end{equation}
Consequently, every sequence $S_j\to\infty$ has a subsequence $R_j$ such that, for almost every $x\in I$ and every $\eps>0$,
\begin{equation}
 |Q_{K,R_j}(x)|\ge R_j^{-d(x)-\eps}
       \quad\text{for all sufficiently large }j.         \label{eq:productlower}
\end{equation}
And this subsequence is independent of $x$.
\end{lm}

\begin{proof}
The second assertion follows from the first and the fact that  $L^1$ convergence implies a.e. convergence along a subsequence. So we only prove the claimed $L^1$ convergence.

Put $g(t)=-\log|\cos(2\pi t)|$, assigning $+\infty$ at its isolated singularities. It is nonnegative and periodic, with
\begin{equation}
 m:=\int_0^1g(t)\dd t=\log2,\quad
 g(t)=O(t^2)\ (t\to0),\quad
 \int_0^1\frac{g(t)}t\dd t<\infty.                       \label{eq:g}
\end{equation}
Let $L_R(x)=\sum_{k\ge1}g(Rx^k)$.

\subsubsection*{Averaging the logarithmic loss}
The function $G(t)=\int_0^t(g(s)-m)\dd s$ is bounded, periodic, and locally absolutely continuous. On the (possibly empty) interval $E_{k,R}:=\{x\in I:Rx^k\ge1\}$, the derivative of $\phi(x)=Rx^k$ is monotone and at least $k/v$. Integration by parts gives
\[
 \int_{E_{k,R}}(g(\phi(x))-m)\dd x
 =\left[\frac{G(\phi(x))}{\phi'(x)}\right]_{\partial E_{k,R}}
   +\int_{E_{k,R}}\frac{G(\phi(x))\phi''(x)}{\phi'(x)^2}\dd x
 =O_I(k^{-1}).
\]

On the complementary region, the substitution $t=Rx^k$ gives
\[
 0\le\int_{I\setminus E_{k,R}}g(Rx^k)\dd x
 \le\frac vk\int_0^1\frac{g(t)}t\dd t=O_I(k^{-1}).
\]
Take $N_R=\lceil\log(8R)/(-\log v)\rceil$. Since $Rv^{N_R}\le1/8$, \eqref{eq:g} bounds $\sum_{k>N_R}g(Rx^k)$ uniformly on $I$ by a constant. The preceding estimates imply
\begin{align}
 \int_I L_R(x)\dd x&=\int_I \sum_{k\ge 1} g(Rx^k)\,dx \\
 &=   \sum_{k=1}^{N_R} \int_{E_{k,R}} (g(Rx^k) - m)\,dx   + m\sum_{k=1}^{N_R}\Leb(E_{k,R})  \\
 &+    \sum_{k=1}^{N_R} \int_{I\setminus E_{k,R}} (g(Rx^k) - m)\,dx +  \sum_{k> N_R} \int_I g(Rx^k)\,dx\\
 &=m\sum_{k=1}^{N_R}\Leb( E_{k,R})+O_I(\log\log R)\notag\\
 &=m\int_I\left\lfloor\frac{\log R}{-\log x}\right\rfloor\dd x
       +O_I(\log\log R)\notag\\
 &=(\log R)\int_Id(x)\dd x+O_I(\log\log R).
                                                        \label{eq:meanloss}
\end{align}
Here in the fourth line we used the fact that $\DS \sum_{j=1}^L \frac{1}{j} \approx \log L$ and that the tail $ \DS \sum_{k> N_R} \int_I g(Rx^k)\,dx$ is $O(1)$. Moreover, in the penultimate equality $\sum_{k=1}^{N_R}\Leb(E_{k,R}) =\int_I\lfloor\frac{\log R}{-\log x}\rfloor \,dx$ holds by Fubini. 

Equation~\eqref{eq:meanloss} immediately implies $\int_I  [\frac{-\log|Q_{K,R}(x)|}{\log R} - d(x)] \,dx\to0$. Below we prove 

\[\label{eq:lowerinmeasure}
 \int_I(d(x)-\frac{-\log|Q_{K,R}(x)|}{\log R})_+\dd x\to0.
\]
which together will imply $L^1$ convergence.

Fix any sequence $\{R_n\}$ increasing to $\infty$ and let $N_n:=\lfloor\frac{\log R_n}{-\log x}\rfloor+1 , x\in I$. Recall Lemma \ref{lem: small tail} , which asserts that for all $\delta>0$, for a.e. $x\in I$, there exists $n_0(x)\in\mathbb{N}$ so that for all $n>n_0$,

\begin{align}
 \nonumber    |\prod_{k=K}^{N_n} \cos(2\pi R_n x^k)| \le 2^{N_n-K+1} R_n^{-\delta},\quad\forall1\le K\le N_n,
\end{align}
where $N_n = \lfloor\frac{\log R_n}{-\log x}\rfloor+1$. The only difference is that previous we show the above for $R_n = b^n$, but now we are dealing with an arbitrary $R_n$ tending to infinity. But the same method of moment bound-Markov inequality work with little modification. In particular, this implies 
\begin{equation}
    (d(x)-\frac{-\log|Q_{K,R_n}(x)|}{\log R_n})_+ \to 0 \;\; a.e.\quad\text{as }n\to\infty. 
\end{equation}
Using bounded convergence theorem, \eqref{eq:lowerinmeasure} is proved with $R$ replaced by $R_n , n\to\infty$. Since $\{R_n\}$ is arbitrary, \eqref{eq:lowerinmeasure} is proved.
\end{proof}

\subsection{A uniform integration-by-parts estimate}\label{subsec: ibp}

This subsection includes a simple integration by parts estimate that will be the core of the proof. 

\begin{lm}\label{lem:ibp}
Let $\chi\in C_c^\infty(U)$, $T\ge1$, and let $\Phi$ be a real smooth function. If, there exist $q\in\mathbb{N}$ and constants $c,C_j,1\le j\le q+1$ such that throughout the support of $\chi$,
\[
 |\Phi'|\ge cT,\qquad |\Phi^{(j)}|\le C_jT
                         \quad(1\le j\le q+1),
\]
then
\[
 \left|\int\chi(x)e^{2\pi\ii\Phi(x)}\dd x\right|
 \le C_q'T^{-q},
\]
where the constant $C_q'$ depends only on $q,\chi,c,C_j$.
\end{lm}

\begin{proof}
Write $\phi=\Phi/T$ and set $A_0=\chi$, $A_{j+1}=-(A_j/\phi')'$. Repeated integration by parts gives
\[
 \int\chi e^{2\pi\ii T\phi}
 =(2\pi\ii T)^{-q}\int A_qe^{2\pi\ii T\phi}.
\]
All boundary terms vanish, and $\|A_q\|_{L^1}$ is bounded in terms of $q,\chi,c,C_j$. This proves the assertion.
\end{proof}

\subsection{ Reduction to the multiplier bound.}
\label{sec: proof, non-diff}

In this subsection we prove the a.e. non-differentiability assertion in Theorem \ref{thm:phase transition}.

Fix $x_0\in G_b$ as defined in Lemma \ref{lem:cutoff}. The strategy is to obtain an estimate of $b^n\Pi_{b^n} f(x_0)$, and show find a neighborhood $J$ of $x_0$ so that for almost every point $x'\in J$, $|b^n\Pi_{b^n} f(x')|$ blows up along a suitably chosen subsequence $b^{n_k}$.

For this fixed $x_0$, take $K,U,\psi$ from Lemma~\ref{lem:cutoff}. Choose an open interval $J\ni x_0$ with $\overline J\subset U$ and $\chi\in C_c^\infty(U)$ equal to one on a neighborhood of $\overline J$. Put $f=\chi h$, $R_n=b^n$. Below $f$ is understood to be $0$ outside $U$, and all compactly supported functions are similarly extended by zero outside its support. $F_\omega$, $\omega\in\Omega$ is evaluated only in $U$.

With these carefully tailored data, we have the following formula for the localization $\Pi_{R_n} f$:

\begin{prop}\label{prop:localization}
For every $L>0$, uniformly in $x\in J$,
\begin{equation}
 \Pi_{R_n}f(x)
 =2^{-K}a^n\cos(2\pi R_n)e^{2\pi\ii R_nS_K(x)}Q_{K,R_n}(x)
     +O_L(R_n^{-L}).                                    \label{eq:localized}
\end{equation}
\end{prop}

Let us assume this Proposition, which will be proved at the end of this subsection, and conclude the proof.

Put
\[
 d_*:=\max_{x\in\overline U}d(x)<1-\alpha,
 \qquad 0<\eps<\frac{1-\alpha-d_*}{3}.
\]
By Lemma~\ref{lem:cosine}, there is an infinite sequence of $n_j\in\mathbb{N}$ such that $|\cos(2\pi R_{n_j})|\ge R_{n_j}^{-\eps}$.
Apply Lemma~\ref{lem:logarithms} on $\overline J$ to this sequence. After taking a further subsequence, for almost every $x\in J$,
\[
 |Q_{K,R_{n_j}}(x)|\ge R_{n_j}^{-d(x)-\eps}
                     \ge R_{n_j}^{-d_*-\eps}
\]
for all sufficiently large $j$. Since $a^n=R_n^{-\alpha}$, Proposition~\ref{prop:localization}, with $L=2$, gives
\[
 R_{n_j}|\Pi_{R_{n_j}}f(x)|
 \ge 2^{-K}R_{n_j}^{1-\alpha-d_*-2\eps}
                  -C R_{n_j}^{-1}\longrightarrow\infty
\]
for almost every $x\in J$. Since $f = h$ on $U$,  Lemma~\ref{lem:pointwise} implies 
$$\limsup_{h\to0} \left|\frac{f(x+h) - f(x)}{h}\right| = \infty$$
at all such $x$.

The intervals $J$ obtained in this way cover $G_b\cap(0,x_c)$. They admit a countable subcover, because
$\R$ is second countable. The union of their exceptional null sets, together with $(0,x_c)\setminus G_b$, is null. This proves the second assertion of Theorem \ref{thm:phase transition}.

\subsection{Proof of the multiplier bound.}
Finally, we prove Proposition \ref{prop:localization}:

Let $f_{m,\omega}(x)=\chi(x)e^{2\pi\ii R_mF_\omega(x)}$, then
\begin{equation}\label{eq: formula of f = chi h}
   \nonumber f(x) =\chi(x)\sum_{m\ge0} a^m \E[e^{2\pi i b^m F_\omega(x)}] =\sum_{m=0}^\infty a^m \E[ f_{m,\omega}(x)].
\end{equation}
Here the expectation is taken over $\omega\in\Omega$ with law $\mathbb{P}$. Taking Fourier transform in $x$, the above becomes
\begin{equation}
    \hat{f}(\xi) = \sum_{m\ge0} a^m \E[\hat{f}_{m,\omega}(\xi)] = \sum_{m\ge 0} a^m \E[\int\chi(x) e^{2\pi i(R_m F_\omega(x) - \xi x)}\,dx].
\end{equation}
Here one can take Fourier transform term by term by uniform convergence. Changing the order of expectation and integration is justified by Fubini. This in turn implies

\begin{equation}\label{eq:formula of cutoff at Rn}
    \mathcal{F}(\Pi_{R_n}f)(\xi) = \sum_{m\ge 0}  a^m 
    \psi\left(\frac{\xi}{R_n}\right)\E[\int\chi(x) e^{2\pi i(R_m F_\omega(x) - \xi x)}\,dx]=: \sum_{m\ge0} F_{m;n}(\xi).
\end{equation}
There will be two cases:

\textbf{Case I}: the terms with $m\ne n$ are negligible. To see this,  note that whenever  $\xi/R_n\in\supp\psi$, \eqref{eq:cross} gives
\[
 |R_mV_\omega(x)-\xi|\ge d(R_m+R_n),\quad \forall x\in U.
\]
Consider $\Phi=R_mF_\omega-\xi x$. Its first order derivative is bounded by the above right hand side, and its $q^{th} , q\ge 2$ derivative is bounded some constant $C_q$ independent of $\xi$, uniformly in $\omega\in\Omega , x\in U$. 
Applying Lemma~\ref{lem:ibp} with $\Phi=R_mF_\omega-\xi x$ and $T=R_m+R_n$ we get
\begin{equation}
 |\widehat f_{m,\omega}(\xi)|
 \le C_q(R_m+R_n)^{-q},                                 \label{eq:offscale}
\end{equation}
and therefore 
\begin{equation}\label{eq: large frequency term}
    F_{m;n}(\xi)\le C_q(R_m+R_n)^{-q}1_{R_n \supp\psi}(\xi).
\end{equation}
Here $1_A$ stands for the indicator function of a set $A$, and $R_n\supp\psi := \{R_n x:x\in \supp\psi\}$.

\textbf{Case II}: The term $m=n$ is can be further divided into a main contribution and a negligible part:  $F_{n;n}(\xi) = M_n(\xi)+T_n(\xi)$. Namely, they are defined as
$$M_n(\xi):= a^n \psi(\frac{\xi}{R_n})\E[1_{\{\omega\in C_k\}}\int\chi(x) e^{2\pi i(R_n F_\omega(x) - \xi x)}\,dx],$$
$$T_n(\xi):= a^n \psi(\frac{\xi}{R_n})\E[1_{\{\omega\notin C_k\}}\int\chi(x) e^{2\pi i(R_n F_\omega(x) - \xi x)}\,dx].$$

To see that $T_n$ is again negligible, note that for words $\omega\notin C_K$,  by Lemma \ref{lem:cutoff},
$|R_nV_\omega(x) - \xi|\ge dR_n$ whenever $\frac{\xi}{R_n}\in \supp\psi$.
So one can apply Lemma \ref{lem:ibp} (with $T = R_n$) to get the bound 
\begin{equation}\label{eq:large frequency term, II}
    T_n(\xi)\le C_qR_n^{-q}1_{R_n \supp\psi}(\xi).
\end{equation}

To estimate $M_n$, we start by showing that the \textit{complementary} multiplier  
$$(1-\psi(\frac{\xi}{R_n}))\E[1_{\{\omega\in C_k\}}\int\chi(x) e^{2\pi i(R_n F_\omega(x) - \xi x)}\,dx]$$
is small. Indeed, for any $\omega\in C_K$, whenever $1-\psi(\xi/R_n)\ne0$, \eqref{eq:pass} implies
\[
 |R_nV_\omega(x)-\xi|\ge d(R_n+|\xi|).
\]
Therefore Lemma~\ref{lem:ibp} implies that, with $T=R_n+|\xi|$, 
\begin{equation}
 |\widehat f_{n,\omega}(\xi)|
       \le C_q(R_n+|\xi|)^{-q}
       \quad\text{on the support of }1-\psi(\xi/R_n).
                                                        \label{eq:complement}
\end{equation}
Hence
\begin{align}\label{eq:main term}
    \nonumber  M_n(\xi) &= a^n\left(\E[1_{\{\omega\in C_k\}}\int\chi(x) e^{2\pi i(R_n F_\omega(x) - \xi x)}\,dx] - (1-\psi(\frac{\xi}{R_n}))\E[1_{\{\omega\in C_k\}}\int\chi(x) e^{2\pi i(R_n F_\omega(x) - \xi x)}\,dx] \right)\\
    &= a^n \int2^{-K}a^n\cos(2\pi R_n)e^{2\pi\ii R_nS_K(x)}Q_{K,R_n}(x) \chi(x) e^{-2\pi i \xi x}\,dx + a^nO_q((R_n+|\xi|)^{-q} ).
\end{align}

Here we used the fact that 
$$\E[1_{\{\omega\in C_k\}} e^{2\pi i R_n F_\omega(x)}] = 2^{-K}\cos(2\pi R_n)e^{2\pi\ii R_nS_K(x)}Q_{K,R_n}(x).$$
Now summing up equations \eqref{eq: large frequency term}, \eqref{eq:large frequency term, II} and \eqref{eq:main term} 
and applying Fourier inversion, we get 

\begin{align*}
  \nonumber  \Pi_{R_n} f(x) &= \chi(x)2^{-K}a^n\cos(2\pi R_n)e^{2\pi\ii R_nS_K(x)}Q_{K,R_n}(x)\\
  &+ O_q(\int (R_n+|\xi|)^{-q}\,d\xi)+ O_q(\sum_{m\neq n}\int (R_m+R_n)^{-q}1_{R_n \supp\psi}(\xi)\,d\xi )\\
  &= 2^{-K}a^n\cos(2\pi R_n)e^{2\pi\ii R_nS_K(x)}Q_{K,R_n}(x) + O_q(|R_n|^{-q+1}).
\end{align*}
Here in the first equation we used Fourier inversion, which is justified since $\chi\in C_c^\infty$, and the last line holds since $\chi\equiv 1$ on $U$. Since $q$ is arbitrary, this proves the proposition.

\printbibliography


\end{document}